\documentclass[10pt]{amsart}

\usepackage[latin9]{inputenc}
\usepackage{color}
\usepackage{amstext}
\usepackage{amsthm}
\usepackage{amssymb}
\usepackage[numbers,sort]{natbib}
\PassOptionsToPackage{normalem}{ulem}
\usepackage{ulem}
\usepackage[unicode=true,
 bookmarks=false,
 breaklinks=false,pdfborder={0 0 1},backref=section,colorlinks=true]
 {hyperref}
\hypersetup{
 linkcolor=red,citecolor=blue,filecolor=dullmagenta,urlcolor=blue}

\makeatletter
\numberwithin{equation}{section}
\numberwithin{figure}{section}

\newtheorem{thm}{Theorem}[section]
\newtheorem{lemma}[thm]{Lemma}
\newtheorem{prop}[thm]{Proposition}

\newtheorem{rem}[thm]{Remark}

\numberwithin{equation}{section}

\newcommand{\les}{\lesssim}
\newcommand{\lam}{{\lambda}}

\newcommand{\ve}{{\varepsilon}}
\newcommand{\de}{{\delta}}

\newcommand{\al}{{\alpha}}

\newcommand{\R}{{\mathbb R}}
\newcommand{\Z}{{\mathbb Z}}

\newcommand{\braxi}{\langle\xi\rangle}

\newcommand{\supp}{{\left(  \textbf{L}, \textbf{N}\right)}}

\def\normo#1{\left\|#1\right\|}

\def\norm#1{\|#1\|}
\def\abs#1{\left|#1\right|}
\def\bra#1{{\langle#1\rangle}}

\def\wt#1{\widetilde{#1}}
\def\wh#1{\widehat{#1}}

\global\long\def\R{\mathbf{\mathbb{R}}}%
\global\long\def\C{\mathbf{\mathbb{C}}}%
\global\long\def\Z{\mathbf{\mathbb{Z}}}%
\global\long\def\N{\mathbf{\mathbb{N}}}%
\global\long\def\jp#1{\langle#1\rangle}%
\global\long\def\norm#1{\|#1\|}%
\global\long\def\wt#1{\widetilde{#1}}%
\global\long\def\ve{\varepsilon}%
\global\long\def\al{\alpha}%
\global\long\def\ep{\epsilon}%

\makeatother

\begin{document}
\title[The modified scattering on 2d semi-relativistic equations]{The modified scattering of 2 dimensional semi-relativistic Hartree equations}
\author{Soonsik Kwon}
\address{Department of Mathematical Sciences, Korea Advanced Institute of Science
and Technology, 291 Daehak-ro, Yuseong-gu, Daejeon 34141, Republic
of Korea}
\email{soonsikk@kaist.edu}
\author{Kiyeon Lee}
\address{Stochastic Analysis and Application Research Center(SAARC), Korea
Advanced Institute of Science and Technology, 291 Daehak-ro, Yuseong-gu,
Daejeon, 34141, Republic of Korea}
\email{kiyeonlee@kaist.ac.kr}
\author{Changhun Yang}
\address{Department of Mathematics, Chungbuk National University, Chungdae-ro1,
Seowon-gu, Cheongju-si, Chungcheongbuk-do, Republic of Korea}
\email{chyang@chungbuk.ac.kr}
\begin{abstract}
In this paper, we consider the asymptotic behaviors of small solutions
to the semi-relativistic Hartree equations in two dimension. The nonlinear
term is convolved with the Coulomb potential $|x|^{-1}$, and it produces
the\emph{ long-range interaction} in the sense of scattering phenomenon.
From this observation, one anticipates that small solutions converge
to a modified scattering states, although they decay as linear solutions.
We show the global well-posedness and the modified scattering for
small solutions in weighted Sobolev spaces. Our proof follows a road
map of exploiting the space-time resonance by \cite{gemasha2008,pusa}.
Compared to the result in three dimensional case \cite{pusa}, weaker
time decay in two dimension is one of the main obstacles.
\end{abstract}

\maketitle

\section{Introduction}
% The subject of this paper is the global well-posedness of 2 dimensional semi-relativistic Hartree equations and an asymptotic behavior of the global solution. In order to handle the nonlinearity that possesses long-range interaction, we give the phase modification which plays a role of getting rid of the resonant interaction and this phase modification yields the modified scattering results. Compared to the 3 dimensional problem \cite{pusa}, weaker time decay leads us main obstacles. To obtain the extra time decay, we exploit the space-time resonant argument. See Sections \ref{sec:main-idea} and \ref{sec:Weighted-Energy-estimate} for the details.
\subsection{The equation and previous results}
We consider the semi-relativistic Hartree equations
with {\em Coulomb} potential
\begin{align}
\left\{ \begin{aligned}-i\partial_{t}u+\sqrt{m^{2}-\Delta}u & =\lam\left(|x|^{-1}*|u|^{2}\right)u\qquad &&\mathrm{in}\;\;\R\times\mathbb{R}^{d},\\
u(0,\cdot) & =u_{0} &&\mathrm{in}\;\;\mathbb{R}^{d},
\end{aligned}
\right.\label{main-eq:semi}
\end{align}
where the unknown $u:\mathbb{R}^{1+d}\to\mathbb{C}$ and some fixed
constant $\lam\in\R$. The nonlocal differential operator $\sqrt{m^{2}-\Delta}$
is defined as a Fourier multiplier operator associated to the symbol
$\sqrt{m^{2}+|\xi|^{2}}$ and $\ast$ denotes the convolution in $\mathbb{R}^d$. Here we consider the mass parameter
$m>0$, so we normalize $m=1$ throughout the paper. For three dimensional case $d=3$, \eqref{main-eq:semi} is often referred to as \textit{Boson star equation} which describes the dynamics and collapse of relativistic Boson stars. It was rigorously derived as the mean-field limit of large systems of bosons. See \cite{Michelangeli2012,frjonlenz2007-nonlinearity,Lieb1987} and references therein. 
% In this paper we focus on two dimensional case $d=2$ and study asymptotic behaviors of solutions.

The mass and energy of solution to \eqref{main-eq:semi} are defined by  
\begin{align}
    M(u)(t)&= \|u(t)\|_{L^2(\R^d)}, \label{mass conservation}\\ 
    E(u)(t)&= \frac12\int_{\R^d}\overline{u}\sqrt{1-\Delta} u dx 
    + \frac{\lambda}{4}\int_{\R^d}\left(|x|^{-1}\ast |u|^2\right) |u|^2 dx \nonumber
\end{align}
respectively and conserved as time evolves. From the conservation laws, one sees that $H^\frac12$ is the energy space. Furthermore,
in massless case $m=0$, we have the scaling symmetry. If $u$ is solution to \eqref{main-eq:semi}, $u_\alpha$ defined by 
$u_\alpha(t,x)=\alpha^{\frac{d}{2}}u(\alpha t,\alpha x)$ are also solutions and the mass is invariant under the scaling, i.e., $\|u(t)\|_{L^2(\R^d)}=\|u_\alpha(t)\|_{L^2(\R^d)}$, thus \eqref{main-eq:semi} is $L^2$-critical.

% The semi-relativistic Hartree equation is written as
% \begin{align}
% -i\partial_{t}u+\sqrt{m^{2}-\Delta}u=\lam\left(V_{\mu}*|u|^{2}\right)u\label{eq:g-semi-relativistic}
% \end{align}
% where\textcolor{red}{{} the potential 
% \begin{align}
% V_{\mu}(x)=\int_{0}^{\infty}e^{-\mu^{2}r-{|x|^{2}}/{4r}}\frac{dr}{r}\sim\left\{ \begin{aligned}\frac{e^{-\mu|x|}}{|\mu x|^{\frac{1}{2}}} & \qquad|x|\gtrsim1,\\
% -\log|x| & \qquad|x|\ll1.
% \end{aligned}
% \right.\label{eq:g-semi-potential}
% \end{align}
% }When $\mu=0$, we refer to that $V_{0}$ is {\em Coulomb} potential
% and when $\mu>0$, we refer to that $V_{\mu}$ is {\em Yukawa}
% potential. Semi-relativistic equation \eqref{eq:g-semi-relativistic}
% is a model to describe the dynamic and collapse of relativistic Boson
% stars. It is often referred to as {\em the Boson star equation}.
% Our main equation \eqref{main-eq:semi} corresponds to the case $V_{0}$
% of \eqref{eq:g-semi-relativistic}.

There are numerous local and global well-posedness results for the semi-relativistic
Hartree equations \eqref{main-eq:semi}. 
A first result was obtained in \cite{lenz2007} where the local well-posedness in $H^{s}(\R^{3})$ for $s\ge\frac{1}{2}$ and global well-posedness in the energy space $H^{\frac12}(\R^{3})$ were established. This result was extended to other dimensions $d\ge2$ in \cite{choz2006-siam}. Also, the authors in \cite{choz2006-siam} established the low regularity well-posedness below the energy space, more precisely, the local well-posedness in $H^s(\R^d)$ for $s>\frac12-\frac{d-1}{4d}$.
This result was later improved in \cite{hele2014} and \cite{lee2021-bkms} where the 
local well-posendess in $H^s(\R^3)$ for $s\ge\frac14$ and $H^s(\R^2)$ for $s>\frac38$ were proved, respectively. 

% extend the result to global in the energy space the global well-posedness for $H^{\frac{1}{2}}(\R^{3})$
% with sufficiently small initial data. The local well-posedness result
% was improved up to \emph{almost optimal regularity} $H^{s}(\R^{3})$
% with $s>\frac{1}{4}$ ($s>0$, when radial symmetric initial data)
% by Herr and Lenzmann \citet{hele2014} via Bilinear estimates and
% $X^{s,b}$-spaces. In two dimensional case, we refer to \citet{lee2021-bkms}.
% As a corollary of theorems in \citet{lee2021-bkms}, we obtain the
% local well-posedness for $H^{s}(\R^{2})$ with $s>\frac{3}{8}$ and
% a lack of uniform continuity of the flow map below $L_{x}^{2}(\R^{2})$.

The aim of this paper is to study the asymptotic behaviors of solutions to \eqref{main-eq:semi} when $d=2$. By a \textit{scattering}, we mean a solution to nonlinear PDEs converges to a solution of the linear equation as time goes to infinity. This phenomenon has been observed in various dispersive equations. Concerning our equation,  let us consider the following generalized model 
\begin{align}\label{eq:gamma}
-i\partial_{t}u+\sqrt{1-\Delta}u =\lam\left(|x|^{-\gamma}*|u|^{2}\right)u,\quad 0<\gamma<d, \quad \mathrm{in}\;\;\R\times\mathbb{R}^{d}.
\end{align}
The asymptotic behavior of solutions to \eqref{eq:gamma} varies depending on the potential, i.e., the range of $\gamma$. To see this, 
by Duhamel's principle, we write \eqref{eq:gamma} as the integral equation
\begin{align*}
u(t) = e^{it\sqrt{1-\Delta}}u_0 + \lam\int_0^t e^{i(t-s)\sqrt{m^2-\Delta}} \left(|x|^{-\gamma}*|u(s)|^{2}\right)u(s)ds.
\end{align*}
Observe that if \eqref{eq:gamma} scatters a linear profile, it would be defined as
\begin{align*}
    u_0+\lim_{t\rightarrow\infty}
    \lam\int_0^t e^{-is\sqrt{1-\Delta}} \left(|x|^{-\gamma}*|u(s)|^{2}\right)u(s)ds.
\end{align*}
By using the well-known time decay estimates of linear solution (see \cite[Lemma~3]{MNO2003})
\begin{align}\label{standard time decay}
    \| e^{it\sqrt{1-\Delta}} u_0 \|_{L^\infty(\R^d)} \les \langle t\rangle^{-\frac d2} \quad \text{for} \quad u_0\in C_0^\infty(\R^d),
\end{align}
one verifies that the time decay of $L^2$ norm of the nonlinearity, computed on a solution to the linear equation, is $t^{-\gamma}$  \footnote{We refer to \cite[Section~4]{choz2006-siam} for the precise statement and detailed proof.}
\begin{align*}
    \normo{(|x|^{-\gamma}*|e^{it\sqrt{1-\Delta}} u_0|^{2})\,e^{it\sqrt{1-\Delta}} u_0}_{L^{2}(\R^d)}\sim |t|^{-\gamma} \quad \text{ for } |t|\gg 1.
\end{align*}
Thus, one may infer that there can not exist a linear profile if $0<\gamma\le1$ with which 
the nonlinearity  is called a \textit{long-range interaction}. Indeed, 
the authors in \cite{choz2006-siam} proved that \eqref{eq:gamma} failed to scatter 
when $0<\gamma\le1$  for  $d\ge3$  and $0<\gamma<\tfrac{d}{2}$  for  $d=1$ or $2$.
On the other hand, for the case $1<\gamma<d$ which is called a \textit{short-range interaction} we may expect the scattering. The first scattering result for the case $2<\gamma<d$ and $d\ge3$ was obtained in \cite{choz2006-siam} and the gap corresponding to $1<\gamma\le 2$ was later closed in \cite{pusa} for $d=3$ and in \cite{hanaog2015-die} for $d\ge3$. 
Recently, scattering result for two dimensional case when $1<\gamma<2$ was established in \cite{YCH}.
Lastly, we refer to \cite{choz2007-jkms,choz2008-dcds-s,chozhishim2009-dcds} for related works.

% In the appendix, we provide the scattering results for the case $d=2$ and $1<\gamma<2$ as a byproduct of main results.

Now, let us focus on the case $\gamma=1$, which corresponds to our main equation. We refer to this as the ``scattering-critical'' case, 
because the time integration barely fails to integrable, or diverges logarithmically. 
We generally anticipate a \textit{modified scattering} result for solutions
to equations which have the scattering critical nonlinearity. The
modified scattering means that a global solution decays as linear
solutions, but converges to a linear solution \emph{with a suitable 
correction} (eg. a phase modification). In the area of nonlinear dispersive
equations, the first modified scattering results was established in \cite{ozawa1991} for one dimensional cubic nonlinear Schr\"odinger equations (NLS). This result was extended to higher
dimension in \cite{hayashi-naumkin1998} and the authors also proved the modified scattering for NLS with Hartree nonlinear terms for $d\ge2$
\begin{align}\label{eq:schrodinger}
    -i\partial_{t}u+\Delta u =\lam\left(|x|^{-1}*|u|^{2}\right)u,\quad \mathrm{in}\;\;\R\times\mathbb{R}^{d}.
\end{align}
Later, in \cite{kapu}, the authors reproved the modified scattering for \eqref{eq:schrodinger}, the same equations addressed in \cite{hayashi-naumkin1998}, by the different technique called \textit{space-time resonance argument} which was introduced in \cite{gemasha2008,gemasha2012-jmpa,gemasha2012-annals}. 
We should mention that the algebraic structure of Schr\"odinger symbol plays a crucial role in their proof. 
Concerning our equation \eqref{main-eq:semi} where the linear operator is nonlocal, the structure of resonance is more involved, so we have to induce a different asymptotic behavior of a solution. Also, relatively higher regularity assumption on initial data is required.

The modified scattering result of \eqref{main-eq:semi} for three dimensional case was proved in \cite{pusa}. We also refer to \cite{hayashi-naumkin2017-henri,iopu2014,sautwang2021-compde} 
where the nonlinear equation with non-local differential operator was studied.
Inspired by work \cite{pusa}, we investigate an asymptotic behavior of solution to \eqref{main-eq:semi} when $d=2$.

\medskip 

\subsection{Main results and ideas}\label{sec:main-idea}
We now state our main theorem for the two dimensional semi-relativistic Hartree equations \eqref{main-eq:semi}:

\begin{thm} \label{main-thm:semi} Let $n\ge1000$ and $k=\frac{n}{100}$.
There exists $\overline{\ve_{0}}>0$ satisfying the following:

Suppose that the initial data $u_{0}$ is sufficiently small in a weighted space.
In other words, for any $\ve_{0}\le\overline{\ve_{0}}$, $u_{0}$
satisfies 
\begin{align}
\|u_{0}\|_{H^{n}(\R^2)}+\|\langle x\rangle^{2}u_{0}\|_{H^{2}(\R^2)}+\|\jp{\xi}^{k}\widehat{u_{0}}\|_{L^{\infty}(\R^2)}\le\ve_{0}.\label{condition-initial:semi}
\end{align}
Then the Cauchy problem \eqref{main-eq:semi} with the initial data $u_{0}$
has a unique global solution $u$ to \eqref{main-eq:semi} decaying
as 
\begin{align}
\|u(t)\|_{L^{\infty}(\R^2)}\les\ve_{0}\bra{t}^{-1}.\label{global-bound:semi}
\end{align}
Moreover, there exists a scattering profile $u_{\infty}$ such that
\begin{align}
\left\Vert \bra{\xi}^{k}\mathcal{F}\left[u(t)-e^{iB(t,D)}e^{-it\bra{D}}u_{\infty}\right]\right\Vert _{L_{\xi}^{\infty}(\R^2)}\les\ve_{0}\bra{t}^{-\de},\label{eq:modified-scattering}
\end{align}
for some $0<\de< \frac{1}{100}$. 
Here, the phase modification is defined by 
\begin{align}\label{formula of B}
    B(t,\xi)=\frac{\lam}{(2\pi)^{2}}\int_{0}^{t}\left( \int_{\R^{2}}\left|\frac{\xi}{\bra{\xi}}-\frac{\sigma}{\bra{\sigma}}\right|^{-1}|\wh{u}(\sigma)|^{2}d\sigma \right) \frac{\rho(s^{-\frac{2}{n}}\xi)}{\bra{s}}ds,    
\end{align}
where $\rho$ is a smooth compactly supported function.
\end{thm}
\begin{rem}
We do not pursue to optimize the regularity
indices $n$ and $k$  and time decay $\delta>0$ in Theorem \ref{main-thm:semi}.
% Throughout this paper, we fix 
% % \begin{equation}
% % n\ge1000,\qquad k=\frac{n}{100}.\label{eq:definition n k}
% % \end{equation}
\end{rem}
\begin{rem}
We prefer to express the formula of phase modification \eqref{formula of B} in the fourier space because it can be seen not only from some heuristic consideration (see \eqref{HD} below) but also in our rigorous proof.

Furthermore, we observe that convergence in the weighted $L^\infty$ norm in \eqref{global-bound:semi} immediately implies the convergence in $L^2$ spaces.
\end{rem}
\begin{rem}
The time decay rate of solutions in \eqref{global-bound:semi} is optimal in the sense that 
the nonlinear solutions decay as the same rate with the linear ones \eqref{standard time decay}.
\end{rem}

Theorem \ref{main-thm:semi} contains the global existence and asymptotic
behaviors of small solutions to \eqref{main-eq:semi}. Our proof to obtain the global existence of solutions is basically based on the bootstrap argument in a weighted Sobolev space, and then the next crucial part is to perform a suitable phase correction and find a modified scattering state.
Briefly, the proof of Theorem \ref{main-thm:semi} consists of threefold.
First, we find the time decay of solutions to \eqref{main-eq:semi}, from which we construct a function space which consists of the weighted energy norm and scattering norm. Then, the second step is to show that the small solutions stay small as long as they exist by performing the weighted energy estimates. Our strategy is based on the method of space-time resonance which was introduced in \cite{gemasha2008,gemasha2012-jmpa,gemasha2012-annals,pusa}. 
The final step is to obtain the bound for the scattering norm in the function space.  It is in this step that a suitable correction of the phase based on the Taylor expansion is required to close the bootstrap argument.
Collecting all from the three steps, we can finally obtain the modified scattering results for \eqref{main-eq:semi}.

Let us explain in detail the ideas of proof in each step.
In the first step, we use the standard stationary phase method on  oscillatory integral to derive the time decay of linear solutions, $t^{-1}$. 
By a direct proof, without resorting to well-known $L^p-L^q$ estimates (e.g \cite[Lemma~3]{MNO2003}), we manage to obtain the time decay of solutions up to higher ($k$ th) order derivative, which is essential in the course of weighted energy estimates to overcome the lack of time decay compared to higher dimensional cases. 
To fully utilize the time decay of solutions, we construct our solution space based on the weighted $L^2$-norms.

In the second step, we show that the small nonlinear solutions stay small during the existing time by performing the weighted energy estimates. We introduce the interaction representation of solutions $u(t)$ so as to track the scattering states 
\begin{align}
f(t,x):=e^{it\langle D\rangle}u(t,x).\label{eq:interation}
\end{align}
Then we can express $f$ via Duhamel's representation
\begin{align}
\begin{aligned}\widehat{f}(t,\xi) & =\widehat{u_{0}}(\xi)+i\lam\mathcal{I}(t,\xi),\\
\mathcal{I}(t,\xi) & = \frac{1}{2\pi}\int_{0}^{t}\int_{\mathbb{R}^{2}\times\mathbb{R}^{2}}e^{is\phi(\xi,\eta)}|\eta|^{-1}\widehat{|u|^2}(\eta)\wh{f}(s,\xi-\eta)d\eta ds
\end{aligned}
\label{eq:duhamel}
\end{align}
with the resonance function 
\begin{align}
\phi(\xi,\eta)=\braxi-\langle\xi-\eta\rangle.\label{eq:resonance-ftn}
\end{align}
In the course of weighted energy estimates, we should bound the $xf$ and $x^2f$ in $L^2$ which are converted to $\nabla \widehat{f}$ and $\nabla^2 \widehat{f}$ in the fourier space, respectively.
The main task is to not only bound the singularity $|\eta|^{-1}$, but also recover the time growth resulting from the derivative $\nabla_\xi$ taken to exponential term. 
Indeed, the most delicate term occurs from $\nabla^2 \widehat{f}$ when two derivatives both fall on $e^{is\phi(\xi,\eta)}$
\begin{align}\label{delicate term}
    \frac{1}{2\pi}\int_{0}^{t}\int_{\mathbb{R}^{2}}s^2\big( \nabla_\xi \phi(\xi,\eta) \big)^2 e^{is\phi(\xi,\eta)}|\eta|^{-1}\widehat{|u(s)|^2}(\eta)\wh{f}(s,\xi-\eta) d\eta ds,
\end{align}
where we have to compensate the time growth $s^{2}$. Here, we encounter the main difficulty from two dimensional nature, i.e., the weaker time decay $|s|^{-1}$ of solutions in contrast to three or higher dimensional problem, 
because $L^{2}$-norm of cubic nonlinearity in \eqref{delicate term} enjoys at most $s^{-2}$ decay which is not sufficient to compensate the time growth for \eqref{delicate term} being integrable in time. 
Nevertheless, since the singularity $|\eta|^{-1}$ near the origin is weaker compared to three or higher dimensional case where $\mathcal{F}(|x|^{-1})(\eta) = C_d |\eta|^{-d+1}$, we anticipate that this advantage leads us to obtaining an extra time decay. One of key observations, as already observed in \cite{pusa}, is the null structure from the phase function
\begin{align}\label{Null structure}
    \nabla_\xi \phi(\xi,\eta) \Big|_{\eta=0}
    =\nabla_\xi \Big( \braxi-\langle\xi-\eta\rangle \Big) \Big|_{\eta=0} = 0.
\end{align}
The null structure removes the singularity near the origin, more precisely, the multiplier in \eqref{delicate term} behaves as 
$$|\nabla_\xi \phi(\xi,\eta)|^2|\eta|^{-1} \sim |\eta|, \; \; \text{ if } \; |\eta|,|\xi|\les 1.$$
Using this, we can heuristically regard $|\eta| \approx |s|^{-1}$ in the analysis respect. Indeed, for $|\eta|\ge |s|^{-1}$, we can exploit the space resonance, in other words, we can apply an integration by parts to the following quadratic term  
\begin{align*}
    \widehat{ |u(s)|^2 } (\eta)
    = \frac{1}{(2\pi)^2}\int_{\R^2} e^{is(\langle \sigma+\eta\rangle - \langle\sigma\rangle )}\widehat{f}(s,\sigma)\overline{\widehat{f}(s,\eta+\sigma}) d\sigma,
\end{align*}
by using $|\nabla_\sigma(\langle \sigma+\eta\rangle - \langle\sigma\rangle )|\sim |\eta|$ for $|\eta|,|\sigma|\les1$ to derive an additional time decay at the cost of $|\eta|^{-1}$.
In the rigorous proof below, we then can control the $xf$ and $x^2f$ in $L^2$, allowing a small growth in $t$.

As mentioned above, since our main equation has the {\em long-range}
nonlinearity, the nonlinearity occurs the logarithm divergence
in terms of time integration. To overcome this difficulty, we employ
a phase modification \eqref{eq:modified-scattering} from the singular potential $|x|^{-1}$ and obtain an extra logarithm time decay by following the argument in \cite{pusa,kapu}. We begin with writing the nonlinear term as 
\begin{align*}
    & \mathcal{I}(t,\xi)=\frac{1}{(2\pi)^3}\int_{0}^{t}\iint_{\R^{2}\times\R^{2}}e^{isp(\xi,\eta,\sigma)}|\eta|^{-1}\wh{f}(s,\xi+\eta)\wh{f}(s,\xi+\sigma)\overline{\wh{f}(s,\xi+\eta+\sigma)}d\eta d\sigma ds,
   \end{align*} 
   where 
   \[
    p(\xi,\eta,\sigma)=\braxi-\langle\xi+\eta\rangle-\langle\xi+\sigma\rangle+\bra{\xi+\eta+\sigma}.
   \]
   Let us assume that $|\xi|\les 1$. By Taylor expansion, the phase function is approximated by
 \begin{align*}
    p(\xi,\eta,\sigma)=\eta\cdot\left(\frac{\xi}{\braxi}-\frac{\xi+\sigma}{\bra{\xi+\sigma}}\right)+O\left(|\eta|^{2}\right).
 \end{align*}  
 Then, neglecting all contributions that decay faster than $|s|^{-1}$, we can approximate the above integration as 
\begin{align}\begin{aligned}\label{HD}
  &\frac{1}{(2\pi)^3}\iint_{|\eta|\lesssim |s|^{-1+}}e^{is\eta\cdot\left(\frac{\xi}{\braxi}-\frac{\xi+\sigma}{\bra{\xi+\sigma}}\right)}|\eta|^{-1}\wh{f}(s,\xi)\wh{f}(s,\xi+\sigma)\overline{\wh{f}(s,\xi+\sigma)}d\eta d\sigma  \\
  &=\frac{1}{(2\pi)^3}\widehat{f}(s,\xi) 
  \iint_{\R^{2}\times\R^{2}} e^{is\eta\cdot\left(\frac{\xi}{\braxi}-\frac{\xi+\sigma}{\bra{\xi+\sigma}}\right)}|\eta|^{-1} d\eta \big|\widehat{f}(s,\xi+\sigma)\big|^2 d\sigma + O(s^{-1-}) \\ 
  &=\frac{1}{2\pi}\widehat{f}(s,\xi)\int_{\R^2} \mathcal{F}^{-1}(|\eta|^{-1}) \left(s\big(\tfrac{\xi}{\braxi}-\tfrac{\sigma}{\bra{\sigma}}\big)\right)\big|\widehat{f}(s,\sigma)\big|^2 d\sigma+ O(s^{-1-}),
\end{aligned}\end{align}
under the suitable assumption on $f$. Then, we obtain 
\begin{align*}
    \partial_t\widehat{f}(t,\xi) = it^{-1}\frac{1}{(2\pi)^2}\widehat{f}(t,\xi)\int_{\R^2}\left|\frac{\xi}{\braxi}-\frac{\sigma}{\bra{\sigma}}\right|^{-1}\big|\widehat{f}(t,\sigma)\big|^2 d\sigma + O(t^{-1-})
\end{align*}
which implies the modified scattering property \eqref{eq:modified-scattering} and \eqref{formula of B}. The rigorous analysis for error terms will be achieved by identifying suitable scale in $\eta$ with respect to time, say $s^{-1+}$, and then, by exploiting the space resonance for $|\eta|\gtrsim s^{-1+}$.

\subsection{Motivation: 2d models}
The two dimensional semi-relativistic equation \eqref{main-eq:semi} might be regarded as a simplified model of the Chern-Simons-Dirac system under the Coulomb gauge condition \footnote{ \, We refer to \cite{bourcanma2014-dcds} for its derivation.}
\begin{align}
    (-i\partial_{t}+\al\cdot D+m\beta)\psi & =N(\psi,\psi)\psi\qquad\mathrm{in}\;\;\R\times\mathbb{R}^{2},\tag{CSD-C}\label{eq:csd-coulomb}
    \end{align}
where the unknown $\psi:\R^{1+2}\to\C^{2}$ and the nonlinear term is given as 
\[
N(\psi,\psi)=\frac{1}{\Delta}\left[\Big(\partial_{1}(\psi^{\dagger}\al^{2}\psi)-\partial_{2}(\psi^{\dagger}\al^{1}\psi)\Big)+\Big(\partial_{2}(|\psi|^{2})\al^{1}-\partial_{1}(|\psi|^{2})\al^{2}\Big)\right]
\]
with Dirac matrices $\al^{j},\beta$ defined as 
\begin{align*}
\al^{1}=\begin{bmatrix}0 & i\\
-i & 0
\end{bmatrix},\quad\al^{2}=\begin{bmatrix}0 & 1\\
1 & 0
\end{bmatrix},\quad\beta=\begin{bmatrix}1 & 0\\
0 & -1
\end{bmatrix}.
\end{align*} 
One of strategy to deal with Dirac operator is, as introduced in \cite{anfosel}, to diagonalize the system using the following identity
\begin{align*}
    \al\cdot D+m\beta = \langle D\rangle \Pi_+(D) - \langle D\rangle \Pi_-(D),
    \end{align*}
where $\Pi_{\pm}(D) = \frac12 \left( I_2\pm \frac{1}{\langle D\rangle}\big[ \alpha\cdot D + \beta\big]\right)$ are the projection operators.
Letting $\psi_{\pm}=\Pi_{\pm}(D)\psi$, \eqref{eq:csd-coulomb} is indeed diagonalized into 
\begin{align*}
    -i\partial_t \psi_{\pm} \pm \langle D\rangle \psi_{\pm}= \sum_{\theta_1,\theta_2,\theta_3\in\{\pm\}}N(\psi_{\theta_1},\psi_{\theta_2})\psi_{\theta_3},
\end{align*}
which consists of the nonlocal differential operator and cubic Hartree nonlinear term as in our main equations, \eqref{main-eq:semi}. 
Especially, the potentials in Hartree term are given as $\frac{\eta_j}{|\eta|^2}$ for $j=1,2$ in the fourier space, which has the similar singularity near the origin and decay property  with order $-1$ as the one given in  \eqref{main-eq:semi}, so \eqref{eq:csd-coulomb} can be also regarded as the scattering critical equation and modified scattering would be expected. 
However, not only long time behaviors but the global existence of solutions to \eqref{eq:csd-coulomb} are still unknown and only the local results have been intensively studied including other choices of gauge \cite{Okamoto2013,bourcanma2014-dcds,Huh2016,Pecher2016}.
One of main difficulty in studying global solutions, compared to our equation \eqref{main-eq:semi}, arises from analysis of the following various resonance functions 
\begin{align}
    p_{(\theta_1,\theta_2,\theta_3)}(\xi,\eta,\sigma)=\braxi-\theta_1\langle\xi-\eta\rangle-\theta_2\langle\eta+\sigma\rangle+\theta_3\langle\sigma\rangle, \quad \theta_i \in \{\pm\}.\label{eq:resonance-ftn}
    \end{align}
Indeed, one can see that the key null structure \eqref{Null structure} to remove  the singularity is no longer valid when $\theta_1=-$.
Nevertheless, we believe that the methodology in this paper with the help of careful analysis of resonance set together with null structures from Dirac operator will play a crucial role in studying the global behavior of solutions to \eqref{eq:csd-coulomb}.

The similar structure also can be observed in the following Dirac equation 
\begin{align}
(-i\partial_{t}+\al\cdot D+m\beta)\psi & =\lam\left(|x|^{-1}*|\psi|^{2}\right)\psi\qquad\mathrm{in}\;\;\R\times\mathbb{R}^{2},\tag{DE}\label{eq:dirac-hartree}
\end{align}
where $\psi:\R^2\rightarrow \C^2$ is the spinor.
\eqref{eq:dirac-hartree} describe the relativistic dynamics of electrons in graphene  and can be derived from the nonlinear Schr\"odinger equation with a potential which is periodic with respect to honeycomb structure \cite{arbuspar2018-jmp}.
 \eqref{eq:dirac-hartree} also can be referred to as the scattering critical equation, but the only local results were established in \cite{hajjmehats2014,lee2021-bkms}. 
 The global well-posedness and modified scattering for \eqref{eq:dirac-hartree}
will be treated in future work. Finally, we refer to \cite{CKLY2022,Cloos2020} for global results for \eqref{eq:dirac-hartree} in three dimension.

\subsection*{Notations}
\; 

\noindent $\bullet$ (Fourier transform)
$\mathcal F g(\xi)=\widehat{g}(\xi):=\int_{\R^2}e^{-ix\cdot\xi}g(x)dx$ and $g(x)=\frac{1}{(2\pi)^2}\int_{\R^2}e^{ix\cdot\xi}\widehat{g}(\xi)d\xi$.

\noindent $\bullet$
$\langle x\rangle:=(1+|x|^{2})^{\frac{1}{2}}$ for $x\in\mathbb{R}^{2}$.

\noindent $\bullet$ (Mixed-normed spaces) For a Banach space $X$
and an interval $I$, $u\in L_{I}^{q}X$ iff $u(t)\in X$ for a.e. $t\in I$
and $\|u\|_{L_{I}^{q}X}:=\|\|u(t)\|_{X}\|_{L_{I}^{q}}<\infty$. Especially,
we denote $L_{I}^{q}L_{x}^{r}=L_{t}^{q}(I;L_{x}^{r}(\mathbb{R}^2))$, $L_{I,x}^{q}=L_{I}^{q}L_{x}^{q}$,
$L_{t}^{q}L_{x}^{r}=L_{\mathbb{R}}^{q}L_{x}^{r}$.

\noindent $\bullet$ As usual, different positive constants are denoted by the same letter $C$, if not specified. $A\lesssim B$ and $A\gtrsim B$ means that
$A\le CB$ and $A\ge C^{-1}B$, respectively for some $C>0$. $A\sim B$
means that $A\lesssim B$ and $A\gtrsim B$.

\noindent $\bullet$ (Fourier multiplier) $D=-i\nabla$. For $m:\R^2\rightarrow\R$, $m(D)f:=\mathcal{F}^{-1}\big( m(\xi)\widehat{f}(\xi) \big)$.

\noindent $\bullet$ (Littlewood-Paley operators) Let $\rho$ be a
bump function such that $\rho\in C_{0}^{\infty}(B(0,2))$
with $\rho(\xi)=1$ for $|\xi|\le1$ and define $\rho_{N}(\xi):=\rho\left(\frac{\xi}{N}\right)-\rho\left(\frac{2\xi}{N}\right)$
for $N\in2^{\mathbb{Z}}$ and also $\rho_{\le N_{0}}:=1-\sum_{N>N_{0}}\rho_{N}$. We define the frequency projection
$P_{N}$ by $\mathcal{F}(P_{N}f)(\xi)=\rho_{N}(\xi)\widehat{f}(\xi)$. In addition
$P_{N_{1}\le\cdot\le N_{2}}:=\sum_{N_{1}\le N\le N_{2}}P_{N}$ and
$P_{\sim N_{0}}:=\sum_{N\sim N_{0}}P_{N}$. For $N\in2^{\mathbb{Z}}$
we denote $\widetilde{\rho_{N}}=\rho_{N/2}+\rho_{N}+\rho_{2N}$. In
particular, $\widetilde{P_{N}}P_{N}=P_{N}\widetilde{P_{N}}=P_{N}$
where $\widetilde{P_{N}}=\mathcal{F}^{-1}\widetilde{\rho_{N}}\mathcal{F}$.
Especially, we denote $P_{N}f$ by $f_{N}$ for any measurable function
$f$.

\noindent $\bullet$ Let $\textbf{v}=(v_{i}),\textbf{w}=(w_{i})\in\R^{2}$.
Then $\textbf{v}\otimes\textbf{w}$ denotes the usual tensor product
such that $(\textbf{v}\otimes\textbf{w})_{ij}=v_{i}w_{j}$ for $i,j=1,2$. We also
denote a tensor product of $\textbf{v}\in\C^{n}$ and $\textbf{w}\in\C^{m}$
by a matrix $\textbf{v}\otimes\textbf{w}=(v_{i}w_{j})_{\substack{i=1,\cdots,n\\
i=1,\cdots,m
}
}$. For simplicity, we use the simplified notation 
\[
\mathbf{v}^{k}=\overbrace{\mathbf{v}\otimes\cdots\otimes\mathbf{v}}^{k\;\text{times}},\qquad\nabla^{k}=\overbrace{\nabla\otimes\cdots\otimes\nabla}^{k\;\text{times}}.
\]
The product of $\mathbf{v}$ and $f\in\C$ is given by $\mathbf{v}f=\mathbf{v}\otimes f$.

\noindent $\bullet$ For the distinction between a vector and a scalar,
we use the bold letter for a vector-valued function and the normal
letter for a scalar-valued function.

%%%%%%%%%%%%%%%%%%%%%%%%%%%%%%%%%%%%%%%%%%%%%%%%%%%%%%%%%%%%%%%%%%%%%%%%%%%%%%%%%%%%%%%%%%%%%%%%%%%%%%%%%%%%%%%%%%%%%%%%%%%%%%%%%%%%%%%%%%%%%%%%%%%%%%%%%%%%%%%%%%

\section{Time decay estimates}
In this section, we find sharp time decay estimates of solutions to linear equation. We define an a priori assumption incorporating time decay to get global solutions to \eqref{main-eq:semi}. 
Define 
\begin{align}
    f(t,x):=e^{it\langle D\rangle}u(t,x),\label{eq:interation}
\end{align}
where $u(t)$ is a solution to \eqref{main-eq:semi}. Let us set
\begin{align}\label{eq:bundle}
	n\ge1000,\;\; k=\frac{n}{100},\;\; \mbox{ and }\;\; \delta_0=\frac{1}{100}.
\end{align} For $\ve_{1}>0$
to be chosen later, we assume a priori smallness of solutions: for
a large time $T>0,$
\begin{align}
\|u\|_{\Sigma_{T}}\les\ve_{1},\label{assumption-apriori}
\end{align}
with 
\begin{align*}
\begin{aligned}\|u\|_{\Sigma_{T}} & :=\sup_{t\in[0,T]}\Big[\bra{t}^{-\de_{0}}\|u(t)\|_{E_{1}}+\bra{t}^{-2\de_{0}}\|u(t)\|_{E_{2}}+\|u(t)\|_{S}\Big],\end{aligned}
\end{align*}
where 
\begin{align*}
\|u(t)\|_{E_{1}} & :=\|u(t)\|_{H^{n}(\R^2)}+\|xe^{it\jp D}u(t)\|_{H^{2}(\R^2)},\\
\|u(t)\|_{E_{2}} & :=\|x^{2}e^{it\jp D}u(t)\|_{H^{2}(\R^2)},\\
\|u(t)\|_{S_{\ }} & :=\left\Vert \braxi^{k}\widehat{u}(t)\right\Vert _{L_{\xi}^{\infty}(\R^2)}.
\end{align*}
We compute the pointwise time decay of the semi-relativistic equation by assuming the a priori assumption \eqref{assumption-apriori}. We refer to \cite[Proposition~3.1]{pusa} for three dimensional case where the sharp time decay $|t|^{-\frac32}$ was obtained. Here, we obtain sharp time decay estimates for two dimensional case by following the similar strategy in \cite[Proposition~3.1]{pusa}, 
but we have to bound the linear solution with the differential order up to $k$.
 \begin{prop}[Time decay]\label{timedecay-prop}
Assume that $u$ satisfies the a priori assumption \eqref{assumption-apriori} for $\ve_{1}$ and $T$ with the index conditions \eqref{eq:bundle}.  Then for small $\ve_{1}$, There exists $C$
satisfying that for $0\le t\le T$ and $0\le\ell\le k$ 
\begin{align}
\|u(t)\|_{W^{\ell,\infty}}\le C\langle t\rangle^{-1}\ve_{1},\label{eq:time-decay}
\end{align}
where the index $k$ is in the a priori assumption \eqref{assumption-apriori}.
\end{prop}

\begin{proof}
We prove that 
\begin{align*}
  \| \bra{D}^{\ell} e^{-it\langle D\rangle} f \|_{L^{\infty}} \lesssim 1,
\end{align*}    
whenever $f$ satisfies
\begin{align}
    \langle t\rangle^{-1}\|\braxi^{k}\widehat{f}\|_{L_{\xi}^{\infty}}+\langle t\rangle^{-1-\de_{0}}\Big[\|xf\|_{L_{x}^{2}}+\|f\|_{H^{n}}\Big]+\bra{t}^{-1-2\de_{0}}\|x^{2}f\|_{L_{x}^{2}}\le\ve_{1}.\label{f-condition}
\end{align}
We begin with writing     
\[
\bra{D}^{\ell}e^{-it\langle D\rangle} f (t,x)=\sum_{N\in2^{\mathbb{Z}}}I_N(t,x)
\]
where 
\begin{align}
  I_N(t,x) &:= \frac{1}{(2\pi)^2}\int_{\mathbb{R}^{2}}\bra{\xi}^{\ell}e^{it\phi(\xi)}\widehat{f}(\xi)\rho_{N}(\xi)d\xi, \nonumber \\ 
    \phi(\xi) &:= -\braxi+\xi\cdot\frac{x}{t}.\label{theta-phase}
\end{align}
% It suffices to prove that a priori assumption \eqref{assumption-apriori}
% implies that for $f(t),$ 
% \begin{align}
% \langle t\rangle^{-1}\|\braxi^{k}\widehat{f}\|_{L_{\xi}^{\infty}}+\langle t\rangle^{-1-\de_{0}}\Big[\|xf\|_{L_{x}^{2}}+\|f\|_{H^{n}}\Big]+\bra{t}^{-1-2\de_{0}}\|x^{2}f\|_{L_{x}^{2}}\le\ve_{1}.\label{f-condition}
% \end{align}
% Under this, it suffices to show 
% \begin{equation}
% \left|\bra{D}^{\ell}u(t,x)\right|\le\sum_{N\in2^{\mathbb{Z}}}I_{N}(t,x)\le C\ve_{1}\label{eq:claim2.1}
% \end{equation}
% where 
% \[
% I_{N}(t,x)=\left|\int_{\mathbb{R}^{2}}\bra{\xi}^{\ell}e^{it\phi(\xi)}\widehat{f}(\xi)\rho_{N}(\xi)d\xi\right|.
% \]
We decompose 
\[
\sum_{N\in2^{\mathbb{Z}}}I_{N}(t,x)=\left(\sum_{N\le\langle t\rangle^{-\frac{1}{2}}}+\sum_{N\ge\langle t\rangle^{\frac{2}{n}}}+\sum_{\langle t\rangle^{-\frac{1}{2}}\le N\le\langle t\rangle^{\frac{2}{n}}}\right)I_{N}(t,x).
\]
The low-frequency part can be estimated as 
\[
\sum_{N\le\langle t\rangle^{-\frac{1}{2}}}I_{N}(t,x)\les\sum_{N\le\langle t\rangle^{-\frac{1}{2}}}\|\rho_{N}\|_{L^{1}}\normo{\bra{\xi}^{\ell}\widehat{f}}_{L_{\xi}^{\infty}}\les\langle t\rangle^{-1}\|\bra{\xi}^{\ell}\widehat{f}\|_{L_{\xi}^{\infty}}\les\ve_{1}.
\]
For the high-frequency, we exploit the high Sobolev norm bound as follows:
\begin{align*}
\sum_{N\ge\langle t\rangle^{\frac{2}{n}}}I_{N}(t,x) & \les\sum_{N\ge\langle t\rangle^{\frac{2}{n}}}\bra{N}^{\ell+1}\|\rho_{N}\widehat{f}\|_{L_{x}^{2}} \les\sum_{N\ge\langle t\rangle^{\frac{2}{n}}}N^{\ell+1-n}\|f\|_{H^{n}}\\
 &  \les\langle t\rangle^{-1-\de_{0}}\|f\|_{H^{n}}\les\ve_{1}. 
\end{align*}

For the remaining mid-frequency part, we apply the non-stationary phase method. %\[
%\sum_{\langle t\rangle^{-\frac{1}{2}}\le N\le\langle t\rangle^{\frac1{15}}}\left|\int_{\mathbb{R}^{3}}e^{it\phi (\xi)}\widehat{f }(\xi)\rho_{N}(\xi)d\xi\right|\les \ve_1.
%\]
One verifies that when $|x|>t$, the phase $\phi$ is non-stationary, i.e. 
\begin{align}
|\nabla_{\xi}\phi(\xi)|\ge \left|\frac{|x|}{t}-\frac{|\xi|}{\braxi}\right|\ge 1-\frac{|\xi|}{\braxi}\gtrsim\braxi^{-2}.\label{non-stat-1}
\end{align}
On the other hand, when $|x|<t$, the phase $\phi$ could be stationary around $\xi_{0}$:
\[
\nabla_{\xi}\phi(\xi_{0})=0\;\;\mbox{where}\;\;\xi_{0}=-\frac{x}{\sqrt{t^{2}-|x|^{2}}}.
\]
We now set $N_0\sim |\xi_{0}|$. 
First, we consider the non-stationary case $N\nsim N_{0}$. Then one can find that  
\begin{align}
\Big|\nabla_{\xi}\phi(\xi)\Big|\gtrsim\max\left(\frac{|\xi-\xi_{0}|}{\bra{N}^{3}},\frac{|\xi-\xi_{0}|}{\bra{N_{0}}^{3}}\right), \text{ for } |\xi|\sim N. \label{non-stat-2}
\end{align}
We perform an integration by parts twice to write $I_N$ as 
\[
\int_{\mathbb{R}^{2}}\bra{\xi}^{\ell}e^{it\phi(\xi)}\widehat{f}(\xi)\rho_{N}(\xi)d\xi=I_{N}^{1}(t,x)+I_{N}^{2}(t,x)+I_{N}^{3}(t,x)
\]
where 
\begin{align}
\begin{aligned}I_{N}^{1}(t,x) & =-t^{-2}\int_{\mathbb{R}^{2}}\bra{\xi}^{\ell}e^{it\phi(\xi)}\frac{\nabla_{\xi}\phi}{|\nabla_{\xi}\phi|^{4}}\cdot\nabla_{\xi}\phi\nabla_{\xi}^{2}\left(\widehat{f_{N}}(\xi)\right)d\xi,\\
I_{N}^{2}(t,x) & =-2t^{-2}\int_{\mathbb{R}^{2}}\bra{\xi}^{\ell}e^{it\phi(\xi)}\nabla_{\xi}\cdot\left(\frac{\nabla_{\xi}\phi}{|\nabla_{\xi}\phi|^{2}}\right)\frac{\nabla_{\xi}\phi}{|\nabla_{\xi}\phi|^{2}}\cdot\nabla_{\xi}\widehat{f_{N}}(\xi)d\xi\\
I_{N}^{3}(t,x) & =-t^{-2}\int_{\mathbb{R}^{2}}\bra{\xi}^{\ell}e^{it\phi(\xi)}\nabla_{\xi}\cdot\left[\nabla_{\xi}\cdot\left(\frac{\nabla_{\xi}\phi}{|\nabla_{\xi}\phi|^{2}}\right)\frac{\nabla_{\xi}\phi}{|\nabla_{\xi}\phi|^{2}}\right]\widehat{f_{N}}(\xi)d\xi.
\end{aligned}
\end{align}
By \eqref{non-stat-1} and \eqref{non-stat-2}, we obtain the following bounds (independent
of $N_{0}$): for $|\xi|\sim N$
\begin{align}
\begin{aligned}
\left|\frac{1}{\nabla_{\xi}\phi(\xi)}\right| & \les N^{-1}\bra{N}^{3},\\
\left|\nabla_{\xi}\cdot\left(\frac{\nabla_{\xi}\phi}{|\nabla_{\xi}\phi|^{2}}\right)\right| & \les N^{-2}\bra{N}^{5},\\
\left|\nabla_{\xi}\cdot\left[\nabla_{\xi}\cdot\left(\frac{\nabla_{\xi}\phi}{|\nabla_{\xi}\phi|^{2}}\right)\frac{\nabla_{\xi}\phi}{|\nabla_{\xi}\phi|^{2}}\right]\right| & \les N^{-4}\bra{N}^{10}.
\end{aligned}
\label{bound-phase}
\end{align}
Using \eqref{bound-phase} and the Sobolev embedding, we see that 
\begin{align}
    \begin{aligned}
    \label{computation}
\Big|I_{N}^{1}(t,x)\Big| & \les t^{-2}N^{-2}\bra{N}^{\ell+6}\left\Vert \nabla_{\xi}^{2}\left(\widehat{f_{N}}\;\right)\right\Vert _{L_{\xi}^{1}}\\
 & \les t^{-2}\bra{N}^{\ell+6}\Big(N^{-1}\|x^{2}f\|_{L^{2}}+N^{-2}\left\Vert \nabla\rho_N\right\Vert _{L^{\frac{4}{3}}}\|\langle x\rangle^{2}f\|_{L^{2}} \\ 
 &\qquad\qquad \qquad\qquad +N^{-2}\min\big( \langle N\rangle^{-n}N^{-1}\| f\|_{H^n}, \langle N\rangle^{-k}\|\wh{f}\|_{L_{\xi}^{\infty}}\big)\Big)\\
 & \les t^{-2}\bra{N}^{\ell+6}\left((N^{-1}+N^{-\frac{3}{2}})t^{1+2\de_{0}}+N^{-2}\min\big(\bra{N}^{-n+1} t^{1+\delta_0},t\big)\right),
    \end{aligned}\end{align}
which implies that 
\[
\sum_{\langle t\rangle^{-\frac{1}{2}}\le N\le\langle t\rangle^{\frac{2}{n}}}\Big|I_{N}^{1}(t,x)\Big|\les\ve_{1}.
\]
$I_{N}^{2}$ can be estimated similarly. Indeed, one has 
\begin{align*}
&|I_{N}^{2}(t,x)|  \les t^{-2}N^{-3}\bra{N}^{\ell+8}\normo{\nabla_{\xi}\left(\rho_{N}\widehat{f}\;\right)}_{L_{\xi}^{1}}\\
&\; \les t^{-2}\bra{N}^{\ell+8}\left(N^{-2}\left\Vert \nabla\rho_N\right\Vert _{L^{\frac{4}{3}}}\|\langle x\rangle^{2}f\|_{L^{2}}+N^{-2}\min\big( \langle N\rangle^{-n}N^{-1}\| f\|_{H^n}, \langle N\rangle^{-k}\|\wh{f}\|_{L_{\xi}^{\infty}}\big)\right),
\end{align*}
which leads us that 
\[
\sum_{\langle t\rangle^{-\frac{1}{2}}\le N\le\langle t\rangle^{\frac{2}{n}}}|I_{N}^{2}(t,x)|\les\ve_{1}.
\]
Lastly, we estimate 
\begin{align*}
    \sum_{\langle t\rangle^{-\frac{1}{2}}\le N\le\langle t\rangle^{\frac{2}{n}}} |I_N^3(t,x)| &\lesssim \sum_{\langle t\rangle^{-\frac{1}{2}}\le N\le\langle t\rangle^{\frac{2}{n}}} t^{-2} N^{-4}\langle N\rangle^{\ell+10} \|\rho_N\widehat{f} \|_{L_{\xi}^1} \\ 
    &\lesssim \sum_{\langle t\rangle^{-\frac{1}{2}}\le N\le\langle t\rangle^{\frac{2}{n}}}t^{-2}\langle N\rangle^{\ell+10}N^{-2}\min\big( \langle N\rangle^{-n}N^{-1}\| f\|_{H^n}, \langle N\rangle^{-k}\|\wh{f}\|_{L_{\xi}^{\infty}}\big)  \\ 
    &\lesssim \ep_1.
\end{align*}

We remain to consider the stationary phase case $N\sim N_{0}$. We
further decompose dyadically the frequency space around $\xi_0$. Let $L_{0}\in2^{\mathbb{Z}}$
such that $\frac{L_{0}}{2}<t^{-\frac{1}{2}}\le L_{0}$. We write 
\[
\left|\int_{\mathbb{R}^{2}}\bra{\xi}^{\ell}e^{it\phi(\xi)}\rho_{N}(\xi)\widehat{f}(\xi)d\xi\right|\le\sum_{L=L_{0}}^{2^{10}N}|J_{L}|
\]
where 
\begin{align*}
J_{L}(t,x)=\left\{ \begin{aligned} & \int_{\mathbb{R}^{2}}\bra{\xi}^{\ell}e^{it\phi(\xi)}\rho_{\le L_{0}}(\xi-\xi_{0})\rho_{N}(\xi)\widehat{f}(\xi)d\xi\qquad\mbox{when }L=L_{0},\\
 & \int_{\mathbb{R}^{2}}\bra{\xi}^{\ell}e^{it\phi(\xi)}\rho_{L}(\xi-\xi_{0})\rho_{N}(\xi)\widehat{f}(\xi)d\xi\quad\qquad\mbox{when }L>L_{0}.
\end{aligned}
\right.
\end{align*}
We bound $J_{L_0}$ the support of which contains the stationary point by just the measure 
\[
|J_{L_{0}}|\les L_{0}^{2}\bra{N}^{l}\|\rho_{N}\widehat{f}\|_{L_{\xi}^{\infty}}\les t^{-1}\|\langle \xi\rangle^k\widehat{f}\|_{L_{\xi}^{\infty}}\les\ve_{1}.
\]
When $L>L_{0}$, we return to the non-stationary phase cases. By
integrating by parts twice, we decompose $J_{L}(t,x)$ into  
\[
J_{L}(t,x)=J_{L}^{1}(t,x)+J_{L}^{2}(t,x)+J_{L}^{3}(t,x),
\]
where 
\begin{align*}
\begin{aligned}J_{L}^{1}(t,x) & =-t^{-2}\int_{\mathbb{R}^{2}}\bra{\xi}^{l}e^{it\phi(\xi)}\frac{\nabla_{\xi}\phi}{|\nabla_{\xi}\phi|^{4}}\cdot\nabla_{\xi}\phi\nabla_{\xi}^{2}\left(\widehat{f}(\xi)\rho_{N}(\xi)\rho_{L}(\xi-\xi_{0})\right)d\xi,\\
J_{L}^{2}(t,x) & =-2t^{-2}\int_{\mathbb{R}^{2}}\bra{\xi}^{l}e^{it\phi(\xi)}\nabla_{\xi}\cdot\left(\frac{\nabla_{\xi}\phi}{|\nabla_{\xi}\phi|^{2}}\right)\frac{\nabla_{\xi}\phi}{|\nabla_{\xi}\phi|^{2}}\cdot\nabla_{\xi}\left(\widehat{f}(\xi)\rho_{N}(\xi)\rho_{L}(\xi-\xi_{0})\right)d\xi,\\
J_{L}^{3}(t,x) & =-t^{-2}\int_{\mathbb{R}^{2}}\bra{\xi}^{l}e^{it\phi(\xi)}\nabla_{\xi}\cdot\left[\nabla_{\xi}\cdot\left(\frac{\nabla_{\xi}\phi}{|\nabla_{\xi}\phi|^{2}}\right)\frac{\nabla_{\xi}\phi}{|\nabla_{\xi}\phi|^{2}}\right]\widehat{f}(\xi)\rho_{N}(\xi)\rho_{L}(\xi-\xi_{0})d\xi.
\end{aligned}
\end{align*}
We estimate $J_{L}^{i}$ for $i=1,2,3$ similarly to the above, but the following bounds are employed instead of \eqref{bound-phase}: for $|\xi|\sim N$ and $|\xi-\xi_0|\sim L$,  
\begin{align*}
\begin{aligned}\left|\nabla_{\xi}\phi(\xi)\right|^{-1} & \les L^{-1}\bra{N}^{3},\\
\left|\nabla_{\xi}\cdot\left(\frac{\nabla_{\xi}\phi}{|\nabla_{\xi}\phi|^{2}}\right)\right| & \les L^{-2}\bra{N}^{5},\\
\left|\nabla_{\xi}\cdot\left[\nabla_{\xi}\cdot\left(\frac{\nabla_{\xi}\phi}{|\nabla_{\xi}\phi|^{2}}\right)\frac{\nabla_{\xi}\phi}{|\nabla_{\xi}\phi|^{2}}\right]\right| & \les L^{-4}\bra{N}^{10}.
\end{aligned}
\end{align*}
The computation as in \eqref{computation} gives the desired results.
We omit the details and complete the proof of \eqref{eq:time-decay}
\end{proof}

Since we have to handle the multipliers to exploit the time decay
\eqref{eq:time-decay} in our main proof, we introduce some useful
estimates in the rest of this section. %\begin{lemma}[Lemma 6.3 of \cite{chleoz}]\label{lem:hls-infty} For
%	any $\mathbb{C}$-valued functions $u_{1},u_{2}\in L_{x}^{2}(\mathbb{R}^{2})\cap L_{x}^{\infty}(\mathbb{R}^{2})$,
%	we get
%	\begin{align*}
%		\||x|^{-1}*(u_{1}\overline{u_{2}})\|_{L_{x}^{\infty}(\mathbb{R}^{2})}\les\|u_{1}\|_{L_{x}^{\frac{4}{1-\ve}}(\mathbb{R}^{2})}\|u_{2}\|_{L_{x}^{\frac{4}{1+\ve}}(\mathbb{R}^{2})}
%	\end{align*}
%	for some $0<\ve\ll1$. \end{lemma}
\begin{lemma}[Coifman-Meyer operator estimates]\label{lem:coif}
Assume that a multiplier $\textbf{m}(\xi,\eta)$ satisfies that 
\begin{align}\label{multiplier bound}
    C_{\mathbf{m}}:=\left\Vert \iint_{\mathbb{R}^{2}\times\R^{2}}\mathbf{m}(\xi,\eta)e^{ix\cdot\xi}e^{iy\cdot\eta}\,d\eta d\xi\right\Vert _{L_{x,y}^{1}(\mathbb{R}^{2}\times\R^{2})}<\infty.    
\end{align}
Then, for $\frac{1}{p}+\frac{1}{q}=\frac{1}{2}$, 
\begin{align}\label{eq:coif-1}
\left\Vert \int_{\mathbb{R}^{2}}\mathbf{m}(\xi,\eta)\widehat{u}(\xi\pm\eta)\widehat{v}(\eta)\,d\eta\right\Vert _{L_{\xi}^{2}(\R^2)}\les C_{\mathbf{m}}\|u\|_{L^{p}}\|v\|_{L^{q}},
\end{align}
and for $\frac{1}{p}+\frac{1}{q}+\frac{1}{r}=1$, 
\begin{align}\label{eq:coif-1-1}
\left|\iint_{\mathbb{R}^{2}\times\R^{2}}\mathbf{m}(\eta,\sigma)\widehat{u}(\eta\pm\sigma)\widehat{v}(\eta)\wh{w}(\sigma)\,d\sigma d\eta\right|\les C_{\mathbf{m}}\|u\|_{L^{p}}\|v\|_{L^{q}}\|w\|_{L^{r}}.
\end{align}
% Moreover, if the multiplier $\textbf{m}(\xi,\eta,\sigma)$ satisfies
% that 
% \[
% \left\Vert \iiint_{\mathbb{R}^{2}\times\R^{2}\times\mathbb{R}^2}\mathbf{m}(\xi,\eta,\sigma)e^{ix\cdot\xi}e^{iy\cdot\eta}e^{iz\cdot\sigma}\,d\sigma d\eta d\xi\right\Vert _{L_{x,y,z}^{1}(\mathbb{R}^{2}\times\R^{2}\times\mathbb{R}^2)}\le C_{\mathbf{m}},
% \]
% for $\frac{1}{p}+\frac{1}{q}+\frac{1}{r}=\frac{1}{2}$, one has 
% \begin{align}
% \left\Vert \iint_{\mathbb{R}^{2}\times\R^{2}}\mathbf{m}(\xi,\eta,\sigma)\widehat{u}(\xi-\eta)\widehat{v}(\eta-\sigma)\wh{w}(\sigma)\,d\sigma d\eta\right\Vert _{L_{\xi}^{2}(\R^2)}\les C_{\mathbf{m}}\|u\|_{L^{p}}\|v\|_{L^{q}}\|w\|_{L^{r}}.\label{eq:coif-2}
% \end{align}
\end{lemma}

\section{\label{sec:Weighted-Energy-estimate}Weighted Energy estimate}
In this section, we prove the energy estimate
which plays an important role in the bootstrap argument. In the following Proposition \ref{prop-energy}, we bound the weighted norms $\|u\|_{E_{1}},\|u\|_{E_{2}}$ in the a priori assumption \eqref{assumption-apriori}.

\begin{prop}[Weighted energy estimate]\label{prop-energy} Assume
that $u\in C([0,T],H^{n})$ satisfies the a priori assumption \eqref{assumption-apriori}
for some $\ve_{1}>0$ with initial data condition \eqref{condition-initial:semi}
for $\ve_{0}>0$. Suppose the index conditions \eqref{eq:bundle}. Then, we have the following estimates
\begin{align}
 & \sup_{t\in[0,T]}\langle t\rangle^{-\de_{0}}\|u(t)\|_{E_{1}}\le\ve_{0}+C\ve_{1}^{3},\label{eq:first-moment}\\
 & \sup_{t\in[0,T]}\langle t\rangle^{-2\de_{0}}\|u(t)\|_{E_{2}}\le\ve_{0}+C\ve_{1}^{3},\label{eq:second-moment}
\end{align}
%\begin{align}
% & \sum_{t\in[0,T]}\langle t\rangle^{-\de_{0}}\|u(t)\|_{H^{n}}\le\ve_{0}+C\ve_{1}^{3},\label{estimate-n}\\
% & \sum_{t\in[0,T]}\langle t\rangle^{-\de_{0}}\|x e^{ it\brad}u(t)\|_{L_x^2}\le\ve_{0}+C\ve_{1}^{3},\label{estimate-1}\\
% & \sum_{t\in[0,T]}\langle t\rangle^{-2\de_0}\|x^2 e^{ it\brad}u(t)\|_{L_x^2}\le\ve_{0}+C\ve_{1}^{3},\label{estimate-2}
%\end{align}
with $\delta_0=\frac{1}{100}$.
\end{prop}

\subsection{Useful inequalities}\label{sub:norm bound}
 For the purpose of proving weighted energy estimates, we introduce some useful inequalities.
\begin{lemma}[Lemma 3.2 in \cite{choz2006-siam}] 
    For any $\mathbb{C}$-valued functions $u\in L^{2}(\R^{2})\cap L^{\infty}(\R^{2})$,
    we get 
    \begin{align}\label{eq:hls}
    \left\||x|^{-1}*(|u|^{2})\right\|_{L^{\infty}(\mathbb{R}^{2})} & \les\|u\|_{L^{2}}\|u\|_{L^{\infty}}.
    \end{align}
    \end{lemma} 
Under the a priori assumption, we find the bounds for the frequency localized terms.
\begin{lemma}\label{lem:norm-esti} Let
    $u$ satisfy the a priori assumption \eqref{assumption-apriori} for some   $\ve_{1}>0$. Suppose the index conditions \eqref{eq:bundle}. Then, for $0\le\gamma\le1$ and a dyadic number $N\in2^{\mathbb{Z}}$,
    \begin{align}\label{PNuinfty}
  \|P_Nu(t)\|_{L^\infty(\mathbb{R}^{2})} \les N^{\gamma}\langle N\rangle^{-k(1-\gamma)}\langle t\rangle^{-(1-\gamma)}\ep_1.
    \end{align}
We also have 
\begin{align}
     \|P_Nf(t)\|_{L^2(\mathbb{R}^{2})}  &\les\min( N^\frac12 \langle t\rangle^{\frac12\delta_0}, \langle N\rangle^{-n})\ep_1  \label{PNxf},\\ 
    \|P_{N}xf(t)\|_{L^2(\mathbb{R}^{2})} &\les \min( N^\frac12 \langle t\rangle^{\frac32\delta_0}, \langle N\rangle^{-2}\langle t\rangle^{\delta_0})\ep_1.  \label{PNx2f}
\end{align}
\end{lemma}
\begin{proof}
By Young's inequality, we have from \eqref{eq:time-decay} that 
\begin{align*}
 \|P_N u(t)\|_{L^\infty(\mathbb{R}^{2})} \le  \left\|\mathcal{F}^{-1}\left( \rho_N\langle \xi\rangle^{-k}\right) \right\|_{L^1}\|u(t)\|_{W^{k,\infty}}
 \les \langle N \rangle^{-k}\langle t\rangle^{-1}\ep_1.
\end{align*} 
On the other hand, interpolating time decay estimates \eqref{eq:time-decay} and the mass conservation law \eqref{mass conservation}, we get for $2\le p\le\infty$, 
    \[
    \|P_Nu(t)\|_{L^{p}(\mathbb{R}^{2})}\les\|P_Nu(t)\|_{L^{\infty}}^{1-\frac{2}{p}}\|u_0\|_{L^{2}}^{\frac{2}{p}}\les\langle N\rangle^{-k(1-\frac2p)}\bra{t}^{-\left(1-\frac{2}{p}\right)}\ve_{1}^{2}.
    \]
Then, by Bernstein's inequality, we obtain for $2\le p<\infty$
\begin{align*}
\|P_N u(t)\|_{L^\infty(\R^2)} \les N^{\frac{2}{p}}\| P_Nu(t)\|_{L^p(\R^2)}
\les N^{\frac{2}{p}}\langle N\rangle^{-k(1-\frac2p)}\langle t\rangle^{-(1-\frac{2}{p})}\ep_1.
\end{align*}

Next, by using Bernstein's inequality and interpolating weighted norms, one can obtain  
\begin{align*}
    \| P_N f\|_{L^2(\R^2)} \les N^\frac12\|  f\|_{L^{\frac43}(\R^2)}
    \les N^\frac12 \|u_0\|_{L^2}^\frac12\|xf\|_{L^2}^\frac12
    \les N^\frac12\langle t\rangle^{\frac12\delta_0}\ep_1^2,
\end{align*}
or, one has
\begin{align*}
    \| P_N f\|_{L^2(\R^2)} \les \langle N\rangle^{-n}\|u\|_{H^n(\R^2)}.
\end{align*}
Then, \eqref{PNxf} follows by interpolating above two estimates.
The last inequality \eqref{PNx2f} can be obtained similarly.
\end{proof}

Next, we consider the quadratic terms.  
%The following lemma is 2 dimensional version of  \cite[Lemma~2.4]{CKLY2022}. 
\begin{lemma}\label{lem:quadratic terms} Let
    $u$ satisfy the a priori assumption \eqref{assumption-apriori} for some
    $\ve_{1}>0$ with the index conditions \eqref{eq:bundle}. For a dyadic number $N\in2^{\Z}$, we have 
    \begin{align}
    \left\Vert P_{N}\Big(|u(t)|^{2}\Big)\right\Vert _{L^{\infty}(\mathbb{R}^{2})} & \les\min(\langle N\rangle^{-k}\langle t\rangle^{-2},N^{2})\ve_{1}^{2},\label{eq:norm-infty}\\
    \left\Vert P_{N}\Big(|u(t)|^{2}\Big)\right\Vert _{L^{2}(\mathbb{R}^{2})} & \les N\langle N\rangle^{-\frac{k}{2}}\ve_{1}^{2}.\label{eq:norm-two}
    \end{align}
\end{lemma}
\begin{proof}
We refer to \cite[Lemma~2.4]{CKLY2022} where the three dimensional case is proved. 
\end{proof}
We observe that no time decay was obtained in \eqref{eq:norm-two}. In the following lemma, 
we find the time decay in \eqref{eq:norm-two} at the cost of derivative. 
Especially, the space resonance is employed to obtain the almost second order time decay $|t|^{-2+}$.
\begin{lemma} Let
    $u$ satisfy the a priori assumption \eqref{assumption-apriori} for some
    $\ve_{1}>0$ with the index conditions \eqref{eq:bundle}. Then, for a dyadic number $L\in2^{\Z}$, we have
\begin{align} 
    \| P_{N}|u(t)|^2 \|_{L^2(\R^2)} &\les \langle t \rangle^{-1+\delta_0}\langle N\rangle^{-k}\ep^2,  \nonumber \\ 
\| P_{N}|u(t)|^2 \|_{L^2(\R^2)} &\les \langle t \rangle^{-2+\frac{3}{2}\delta_0}N^{-1}\langle N\rangle^{-1}\ep^2. \label{ineq:PLu2}
\end{align}
\end{lemma}
\begin{proof}The Plancherel's identity yields that 
    \begin{align*}
        \| P_{N}|u(t)|^2 \|_{L^2(\R^2)}
        = \| \rho_N \mathcal{F} \big( |u(t)|^2\big) \|_{L^2(\R^2)}.
    \end{align*}
We write 
 \begin{align*}
    \rho_N(\eta)\mathcal{F} \big( |u(t)|^2\big) (\eta) &= \frac{1}{(2\pi)^2}\int_{\R^2}\rho_N(\eta) \widehat{u}(\sigma)\overline{\widehat{u}(\sigma+\eta)} d\sigma \\
    &=\frac{1}{(2\pi)^2}\int_{\R^2} \frac{\rho_N(\eta)}{\langle \sigma\rangle^k\langle\sigma+\eta\rangle^k}\widehat{\langle D\rangle^k u}(\sigma)\overline{\widehat{\langle D\rangle^k u}(\sigma+\eta)} d\sigma. 
\end{align*}
If we let 
$$
\mathbf{m}(\eta,\sigma):=\frac{\rho_N(\eta)}{\langle \sigma\rangle^k\langle\sigma+\eta\rangle^k},
$$ 
 one can verify that 
\begin{align*}
    C_{\mathbf{m}} \lesssim \langle N\rangle^{-k},
\end{align*}
where the constant $C_{\mathbf{m}}$ is defined in \eqref{multiplier bound}. Thus, by the Coifman-Meyer estimates \eqref{eq:coif-1}, we have 
\begin{align*}
    \| \rho_N\mathcal{F} \big( |u(t)|^2\big) \|_{L^2(\R^2)}
    \lesssim \langle N\rangle^{-k}\|u\|_{H^k}\|\langle D\rangle^{k} u\|_{L^\infty}  
    \lesssim \langle N\rangle^{-k} \langle t\rangle^{-1+\delta_0}\ep_1^2.
\end{align*}

%  We begin with the proof of \eqref{ineq:PLu2}.  The Plancherel's identity yields that 
%  \begin{align*}
%      \| P_{L}|u(t)|^2 \|_{L^2(\R^2)}
%      = \| \rho_L \mathcal{F} \big( |u(t)|^2\big) \|_{L^2(\R^2)}.
%  \end{align*}

 Next, we consider \eqref{ineq:PLu2}. we write 
\begin{align}\label{FT |u|^2}
    \mathcal{F} \big( |u(t)|^2\big) (\eta)
    = \frac{1}{(2\pi)^2}\int_{\R^2} e^{it(\langle \sigma+\eta\rangle - \langle\sigma\rangle )}\widehat{f}(t,\sigma)\overline{\widehat{f}(t,\eta+\sigma}) d\sigma,
\end{align}
where $f(t,x)=e^{it\langle D\rangle}u(t,x)$.
We perform an integration by parts to obtain 
\begin{align*}
    \mathcal{F} \big( |u(t)|^2\big) (\eta)
    =\frac{1}{(2\pi)^2}\Big( I_1(t,\eta)+I_2(t,\eta)+I_3(t,\eta) \Big),
\end{align*}
where
\begin{align*}
    I_1(t,\eta)&=\frac{i}{t}\int_{\R^2} \frac{\nabla_\sigma(\langle \sigma+\eta\rangle - \langle\sigma\rangle)}{|\nabla_\sigma(\langle \sigma+\eta\rangle - \langle\sigma\rangle)|^2}e^{-it\langle\sigma\rangle} \cdot  \widehat{xf}(t,\sigma)\overline{\widehat{u}(t,\eta+\sigma}) d\sigma , \\
    I_2(t,\eta)&=\frac{i}{t}\int_{\R^2} \frac{\nabla_\sigma(\langle \sigma+\eta\rangle - \langle\sigma\rangle)}{|\nabla_\sigma(\langle \sigma+\eta\rangle - \langle\sigma\rangle)|^2}e^{it\langle \sigma+\eta\rangle } \cdot  \widehat{u}(t,\sigma)\overline{\widehat{xf}(t,\eta+\sigma})  d\sigma, \\ 
    I_3(t,\eta)&= \frac{i}{t}\int_{\R^2} \nabla_\sigma \cdot \left(  \frac{\nabla_\sigma(\langle \sigma+\eta\rangle - \langle\sigma\rangle)}{|\nabla_\sigma(\langle \sigma+\eta\rangle - \langle\sigma\rangle)|^2} \right) \widehat{u}(t,\sigma)\overline{\widehat{u}(t,\eta+\sigma}) d\sigma.
\end{align*}

% \begin{align*}
%   \mathcal{F} \big( |u(t)|^2\big) (\eta)
%     &= I_1+I_2+I_3, \\ 
%     \text{ where } I_1&=\frac{i}{t}\int_{\R^2} \frac{\nabla_\sigma(\langle \sigma+\eta\rangle - \langle\sigma\rangle)}{|\nabla_\sigma(\langle \sigma+\eta\rangle - \langle\sigma\rangle)|^2}e^{it(\langle \sigma+\eta\rangle - \langle\sigma\rangle )} \cdot \nabla_\sigma \Big\{ \widehat{f}(t,\sigma)\overline{\widehat{f}(t,\eta+\sigma}) \Big\} d\sigma \\
%     &\qquad + \frac{i}{t}\int_{\R^2} \nabla_\sigma \cdot \left(  \frac{\nabla_\sigma(\langle \sigma+\eta\rangle - \langle\sigma\rangle)}{|\nabla_\sigma(\langle \sigma+\eta\rangle - \langle\sigma\rangle)|^2} \right) e^{it(\langle \sigma+\eta\rangle - \langle\sigma\rangle )}\widehat{f}(t,\sigma)\overline{\widehat{f}(t,\eta+\sigma}) d\sigma 
% \end{align*}

First, we consider $I_1$. We perform dyadic decomposition in the variables $\sigma$ and $\eta+\sigma$ to write 
\begin{align}\label{I1dyadic}
    \rho_N (\eta)I_1(t,\eta)&=\frac{i}{t}\sum_{(N_1,N_2)\in(2^{\Z})^2} I_1^{(N,N_1,N_2)}(t,\eta), \\ 
    I_1^{(N,N_1,N_2)}(t,\eta)&:=\frac{i}{t}\int_{\R^2}\mathbf{m}_{(N,N_1,N_2)}(\eta,\sigma)e^{-it\langle \sigma\rangle} \widehat{P_{N_1}xf}(t,\sigma) \overline{\widehat{P_{N_2}u}(t,\eta+\sigma}) d\sigma, \nonumber
\end{align}
where 
\begin{align*}
    \mathbf{m}_{(N,N_1,N_2)}(\eta,\sigma)=\frac{\nabla_\sigma(\langle \sigma+\eta\rangle - \langle\sigma\rangle)}{|\nabla_\sigma(\langle \sigma+\eta\rangle - \langle\sigma\rangle)|^2}\rho_N (\eta)\rho_{N_1} (\eta+\sigma)\rho_{N_2} (\sigma). 
\end{align*}
Since
\begin{align*}
    \left| \nabla_\sigma(\langle \sigma+\eta\rangle - \langle\sigma\rangle)\right|
  =\left|\frac{\eta+\sigma}{\bra{\eta+\sigma}}-\frac{\sigma}{\bra{\sigma}}\right|
    \gtrsim \frac{|\eta|}{\max(\langle \eta+\sigma\rangle, \langle\sigma\rangle) \min(\langle \eta+\sigma\rangle, \langle\sigma\rangle)^2}, 
\end{align*}
a direct computation yields that \footnote{For detailed computation, we refer to \cite{CKLY2022}.} 
\begin{align*}
  |\mathbf{m}_{(N,N_1,N_2)}(\eta,\sigma)| \les N^{-1}\max(\langle N_1\rangle, \langle N_2\rangle) \min(\langle N_1\rangle, \langle N_2\rangle)^2 
\end{align*}
and 
\begin{align}\label{cmbound}
    C_{\mathbf{m}_{(N,N_1,N_2)}} \les N^{-1}\max(\langle N_1\rangle, \langle N_2\rangle) \min(\langle N_1\rangle, \langle N_2\rangle)^{10}.
\end{align}
% From now, we denote $\max(\langle N_1\rangle, \langle N_2\rangle)$ and $\min(\langle N_1\rangle, \langle N_2\rangle)$ by $N_{min}$ and $N_{max}$, respectively.
Applying the operator inequality \eqref{eq:coif-1} with \eqref{cmbound}, we obtain 
\begin{align*}
    & \left\|I_1^{(N,N_1,N_2)}(t)\right\|_{L^2(\R^2)} \\ 
    &\qquad \les |t|^{-1}   N^{-1}\max(\langle N_1\rangle, \langle N_2\rangle) \min(\langle N_1\rangle, \langle N_2\rangle)^{10}
    \|P_{N_1}xf(t)\|_{L^2} \|P_{N_2}u(t)\|_{L^\infty}
\end{align*}
We observe that the sums in \eqref{I1dyadic} are actually taken over those indexes $(N_1,N_2)$ satisfying  
\begin{align*}
    N\les N_1\sim N_2 \; \text{ or } \; N_{\min}\ll N_{\max} \sim N,
\end{align*}
where $N_{\max} = \max(N_1, N_2)$ and $N_{\min} = \min(N_1, N_2)$. Thus, using Lemma~\ref{lem:norm-esti}, we estimate
\begin{align*}
 &\sum_{ N_1 \les N_2} \left\| I_1^{(N,N_1,N_2)}(t) \right\|_{L^2(\R^2)} \\    
 &\les |t|^{-1} N^{-1}  
 \Big( \sum_{N_{1}\ll N_2\sim N}  + \sum_{ N\les N_1\sim N_2} \Big)\langle N_2\rangle^{1-k} \langle N_{1}\rangle^{10} \min( N_1^\frac12 \langle t\rangle^{\frac32\delta_0}, \langle N_1\rangle^{-2} \langle t\rangle^{\delta_0}) \langle t\rangle^{-1}\ep_1^2 \\ 
 &\les \langle t\rangle^{-2+\frac32\delta_0}N^{-1}\langle N\rangle^{-1}
\end{align*}
and 
\begin{align*}
	\sum_{N_{2}\ll N_{1}\sim N} \left\| I_1^{(N,N_1,N_2)}(t) \right\|_{L^2(\R^2)}  & \les |t|^{-1}N^{-1}   
\sum_{N_{2}\ll N_{1}\sim N}  \langle N_1\rangle^{-1} \langle N_2\rangle^{10-k+k\frac{\delta_0}{2}}  \langle t\rangle^{\delta_0} N_2^{\frac{\delta_0}{2}}\langle t\rangle^{-1+\frac{1}{2}\delta_0}\ep_1^2 \\ 
&\quad \les \langle t\rangle^{-2+\frac{3}{2}\delta_0}N^{-1}   \langle N\rangle^{-1} 
\sum_{N_{2}\ll  N}  \langle N_2\rangle^{10-k+k\frac{\delta_0}{2}}  N_2^{\frac{\delta_0}{2}}\ep_1^2 \\ 
&\quad  \les \langle t\rangle^{-2+\frac32\delta_0}N^{-1}\langle L\rangle^{-1}.
\end{align*}
 By the symmetry, we may omit the proof of estimates for $I_2$.

For $I_3$, we also perform dyadic decomposition and write 
\begin{align*}
    \rho_N (\eta)I_3(t,\eta)&=\frac{i}{t}\sum_{N_1,N_2\in(2^{\Z})^2} I_3^{(N,N_1,N_2)}(t,\eta), \\ 
    I_3^{(N,N_1,N_2)}(t,\eta)&=\frac{i}{t}\int_{\R^2}\mathbf{m'}_{(N,N_1,N_2)}(\eta,\sigma)\widehat{P_{N_1}u}(t,\sigma) \overline{\widehat{P_{N_2}u}(t,\eta+\sigma}) d\sigma,
\end{align*}
where 
\begin{align*}
    \mathbf{m'}_{(N,N_1,N_2)}(\eta,\sigma)=\nabla_{\sigma}\cdot \left( \frac{\nabla_\sigma(\langle \sigma+\eta\rangle - \langle\sigma\rangle)}{|\nabla_\sigma(\langle \sigma+\eta\rangle - \langle\sigma\rangle)|^2}\right)\rho_N(\eta)\rho_{N_1} (\eta+\sigma)\rho_{N_2} (\sigma). 
\end{align*}
Then, one easily verifies that 
\begin{align*}
    |\mathbf{m'}_{(N,N_1,N_2)}(\eta,\sigma)| \les N^{-1}\max(\langle N_1\rangle, \langle N_2\rangle)\min(\langle N_1\rangle, \langle N_2\rangle)^3 
  \end{align*}
  and 
  \begin{align}\label{cmbound2}
      C_{\mathbf{m'}_{(N,N_1,N_2)}} \les N^{-1} \max(\langle N_1\rangle, \langle N_2\rangle)\min(\langle N_1\rangle, \langle N_2\rangle)^{11}.
  \end{align}
Applying the operator inequality \eqref{eq:coif-1} with \eqref{cmbound2}, we obtain 
\begin{align*}
	\begin{aligned}
		&\sum_{(N_1,N_2)\in(2^{\Z})^2}  \left\| I_3^{(N,N_1,N_2)}(t) \right\|_{L^2(\R^2)} \\ 
		&\;\les |t|^{-1}  N^{-1} \sum_{(N_1,N_2)\in(2^{\Z})^2}  \max(\langle N_1\rangle, \langle N_2\rangle)\min(\langle N_1\rangle, \langle N_2\rangle)^{11}
		\|P_{N_1}u\|_{L^2} \|P_{N_2}u(t)\|_{L^\infty}\\
    &\; \les |t|^{-2+\frac32\delta_0}N^{-1}\sum_{(N_1,N_2)\in(2^{\Z})^2}\
    \max(\langle N_1\rangle, \langle N_2\rangle)\min(\langle N_1\rangle, \langle N_2\rangle)^{11}
    N_1^\frac12  \langle N_1\rangle^{-n+\frac12} 
    N_2^{\delta_0}\langle N_2\rangle^{-k+k\delta_0}\ep_1^2 \\ 
    &\; \les \langle t\rangle^{-2+\delta_0}N^{-1}\langle N\rangle^{-1}.
\end{aligned}
\end{align*} 
Here we used Lemma~\ref{lem:norm-esti} in the second inequality. This finishes the proof.
\end{proof}

%%%%%%%%%%%%%%%%%%%%%%%%%%%%%%%%%%%%%%%%%%%%%%%%%%%%%%%%%%%%%%%%%%%%%%%%%%%%%%%%%%%%%%%%%%%%%%%%%%%%%%%%%%%%%%%%%%%%%%%%%%%%%%%%%%%%%%%%%%%%%%%%%%%%%%%%%%%%%%%%%%%%%%%%%%%%%%%%%%%%%%%%%%%%%%%%%%%%%%%%%%%%%%%%%%%%%%%%%%%%%%%%%%%%%%%%%%%%%%%%%%%%%%%%%%%%%%%%%%%%%%%%%%%%%%%%%%%%%%%%%%%%%%%%%%%%%%%%%%%%%%%%%%%%%%%%%%%%%%%%%%%%%%%%%%

\subsection{Proof of Proposition~\ref{prop-energy}} Now, we begin with the proof of Proposition~\ref{prop-energy}. 
\begin{proof}[Proof of \eqref{eq:first-moment}]
Let us first handle the high Sobolev norm in $\|u(t)\|_{E_{1}}$.
This can be bounded by Hardy-Littlewood-Sobolev inequality and \eqref{eq:hls}. Indeed, we first observe that by interpolation of the time decay estimates \eqref{eq:time-decay} and the conservation law \eqref{mass conservation}, we get for $2\le p\le\infty$, 
\[
\|u(t)\|_{L_{x}^{p}(\R^2)}\les\|u(t)\|_{L_{x}^{\infty}}^{1-\frac{2}{p}}\|u(t)\|_{L_{x}^{2}}^{\frac{2}{p}}\les\bra{t}^{-\left(1-\frac{2}{p}\right)}\ve_{1}^{2}.
\]
Then, we estimate 
\begin{align*}
\normo{|x|^{-1}*|u(t)|^{2}}_{L_{x}^{\infty}(\R^2)}\les\|u_{0}\|_{L_{x}^{2}}\|u(t)\|_{L_{x}^{\infty}}\les\bra{t}^{-1}\ve_{1}^{2}
\end{align*}
and 
\begin{align*}
\normo{\bra{D}^{n}\left(|x|^{-1}*|u(t)|^{2}\right)}_{L_x^4(\R^2)} & \les\|u(t)\|_{H^{n}}\|u(t)\|_{L_{x}^{4}}\les\bra{t}^{-\frac{1}{2}+\de_{0}}\ve_{1}^{2},
\end{align*}
which imply 
\[
\|u(t)\|_{H^{n}(\R^2)}\les\ve_{0}+C\bra{t}^{\de_{0}}\ve_{1}^{3}.
\]

Let us consider $\|xe^{it\bra{D}}u\|_{H^{2}(\R^2)}$. Note that $$\|xe^{it\jp D}u\|_{H^{2}}\sim\left\|\langle\xi\rangle^2\mathcal{F}\left(xe^{it\jp D}u\right)\right\|_{L_{\xi}^{2}}\sim\left\Vert \langle\xi\rangle^2\nabla_{\xi}\widehat{f}\,\right\Vert _{L_{\xi}^{2}}.$$
By the Duhamel's formula \eqref{eq:duhamel},  $\nabla_{\xi}\widehat{f}$ can be represented by 
\begin{align*}
\nabla_{\xi}\widehat{f}(t,\xi) & =\nabla_{\xi}\widehat{u_{0}}(\xi)+\frac{i\lam}{2\pi}\int_0^t\Big[\mathcal{I}^{1}(s,\xi)+\mathcal{I}^{2}(s,\xi)\Big]ds,
\end{align*}
where 
\begin{align*}
\mathcal{I}^{1}(s,\xi) & =\int_{\mathbb{R}^{2}}e^{is{\phi}(\xi,\eta)}|\eta|^{-1}\nabla_{\xi}\widehat{f}(\xi-\eta)\mathcal{F}(|u|^{2})(\eta)\,d\eta,  \\
\mathcal{I}^{2}(s,\xi) & =is\int_{\mathbb{R}^{2}}\nabla_{\xi}\phi(\xi,\eta)e^{is\phi(\xi,\eta)}|\eta|^{-1}\widehat{f}(\xi-\eta)\mathcal{F}(|u|^{2})(\eta)\,d\eta.
\end{align*}
Here we defined a resonant function $\phi$:
\begin{align} \label{eq:pm}\phi(\xi,\eta) & =\langle\xi\rangle-\langle\xi-\eta\rangle, \;\;\text{ and }\;\;
 \nabla_{\xi}\phi(\xi,\eta)=\frac{\xi}{\langle\xi\rangle}-\frac{\xi-\eta}{\langle\xi-\eta\rangle}.
\end{align}
We estimate the contribution from $\mathcal{I}^{1}$ and $\mathcal{I}^{2}$
under the a priori assumption \eqref{assumption-apriori} as follows:
\[
\normo{\langle\xi\rangle^2\mathcal{I}^{1}(s,\xi)}_{L_{\xi}^{2}}+\normo{\langle\xi\rangle^2\mathcal{I}^{2}(s,\xi)}_{L_{\xi}^{2}}\les\langle s\rangle^{-1+\de_{0}}\ve_{1}^{3}.
\]
\noindent \uline{Estimates for \mbox{$\mathcal{I}^{1}$}}. By Lemma \ref{lem:norm-esti}
and a priori assumption \eqref{assumption-apriori}, we get 
\begin{align}\begin{aligned}\label{I1}
\normo{\langle\xi\rangle^2\mathcal{I}^{1}(s,\xi)}_{L_{\xi}^{2}} 
& \les  \| u(s)\|_{L^\infty} \| u(s)\|_{H^2}\|xf(s)\|_{H^2}   \les \bra{s}^{-1+\de_{0}}\ve_{1}^{3}.
\end{aligned} \end{align}

\noindent \uline{Estimates for \mbox{$\mathcal{I}^{2}$}}. We decompose the frequency variables $|\xi|,|\xi-\eta|$ into the dyadic pieces $N_{0},N_{1}\in2^{\Z}$, respectively.
We also divide $|\eta|$ associated to the potential into $N_2\in2^{\Z}$. Then, we write 
\begin{align*}
    \langle \xi\rangle^2\mathcal{I}^{2}(s,\xi)&=\sum_{\mathbf{N}:=(N_{0},N_{1},N_2)\in (2^{\Z})^3}\mathcal{I}_{\mathbf{N}}^{2}(s,\xi),\\
 \mathcal{I}_{\mathbf{N}}^{2}(s,\xi) &=is\int_{\mathbb{R}^{2}}\mathbf{m}_{\mathbf{N}}(\xi,\eta)e^{is\phi(\xi,\eta)}\widehat{P_{N_1}f}(\xi-\eta)\widehat{P_{N_2}(|u|^{2})}(\eta)\,d\eta,
\end{align*}
where 
\begin{align*}
    \mathbf{m}_{\mathbf{N}}(\xi,\eta) 
    = \langle \xi\rangle^2 |\eta|^{-1} \nabla_{\xi}\phi(\xi,\eta)\rho_{N_{0}}(\xi)\rho_{L}(\eta)\rho_{N_{1}}(\xi-\eta).
\end{align*}
We observe that the sums are actually taken over those indices $(N_{0},N_{1},N_2)$ in the following set 
\begin{align*}
 \mathcal{N}:= \left\{ (N_{0},N_{1},N_2)\in (2^{\Z})^3 \; | \; N_0\les N_1\sim N_2 \text{ or } N_0\sim \max(N_1,N_2)\right\}.
\end{align*}
Indeed,  the integral $\mathcal{I}_{\mathbf{N}}^{2}$ is zero for $(N_{0},N_{1},N_2)\notin \mathcal{N}$.

We can estimate $\mathcal{I}_{\mathbf{N}}^{2}(s)$ in two ways. 
First, we have by H\"older inequality
\begin{align}\label{I2:holder}
    \normo{\mathcal{I}_{\textbf{N}}^{2}(s,\xi)}_{L_{\xi}^{2}} 
\les |s|\normo{\textbf{m}_{\mathbf{N}}(\xi,\eta)}_{L_{\xi,\eta}^{\infty}}\|\rho_{N_{0}}\|_{L^{2}}\normo{P_{N_2}|u(s)|^{2}}_{L^{2}}\normo{P_{N_{1}}{f}(s)\,}_{L^{2}}.
\end{align}
On the other hand, by using the operator inequality \eqref{eq:coif-1} with $C_{\mathbf{m}_{\mathbf{N}}}$ satisfying
\begin{align*}
    C_{\mathbf{m}_{\mathbf{N}}}:= \left\Vert \iint_{\mathbb{R}^{2}\times\R^{2}}\mathbf{m}_{\mathbf{N}}(\xi,\eta)e^{ix\cdot\xi}e^{iy\cdot\eta}\,d\eta d\xi\right\Vert _{L_{x,y}^{1}(\mathbb{R}^{2}\times\R^{2})}<\infty ,
\end{align*}
we have 
\begin{align}\label{I2:CM}
    \normo{\mathcal{I}_{\textbf{N}}^{2}(s,\xi)}_{L_{\xi}^{2}} 
    \les  |s|C_{\mathbf{m}_{\mathbf{N}}}\normo{P_{N_2}|u(s)|^{2}}_{L^{\infty}}\normo{P_{N_{1}}{f}(s)\,}_{L^{2}}.
\end{align}
% Let us recall from Lemma~\ref{lem:quadratic terms} that 
% \begin{align*}
%     \left\Vert P_{N}\Big(|u(t)|^{2}\Big)\right\Vert _{L^{\infty}} & \les\min(\langle N\rangle^{-k}\langle t\rangle^{-2},N^{2})\ve_{1}^{2}, \\
%     \left\Vert P_{N}\Big(|u(t)|^{2}\Big)\right\Vert _{L^{2}} & \les N\langle N\rangle^{-\frac{k}{2}}\ve_{1}^{2},
% \end{align*}
% We see from the a priori assumption \eqref{assumption-apriori} that 
% \begin{align*}
%     \normo{P_{N_{1}}{f}(s)\,}_{L^{2}} 
%     \les N_1\langle N_1\rangle^{-k}\ep_1.
% \end{align*}
% In addition, 
From the following inequality
\begin{align}\label{bound of grad phi}
   \left|  \nabla_{\xi}\phi(\xi,\eta) \right| =\left|  \frac{\xi}{\langle\xi\rangle}-\frac{\xi-\eta}{\langle\xi-\eta\rangle} \right| 
   \les \frac{|\eta|}{\max(\langle\xi\rangle,\langle\xi-\eta\rangle)},
\end{align}
one can readily verify that 
\begin{align}\label{pointwise bound of m}
    \sup_{\xi,\eta\in\R^2}\left|\mathbf{m}_{\mathbf{N}}(\xi,\eta)\right|\les \langle N_0\rangle^2\max(\langle N_0\rangle, \langle N_1\rangle)^{-1}  \; \text{ and }  \;  C_{\mathbf{m}_{\mathbf{N}}}\les  \langle N_0\rangle^2\max(\langle N_0\rangle, \langle N_1\rangle)^{-1} .
\end{align} 
Now, we estimate the sum over those indexes $N_0$ such that $N_0\le\langle s\rangle^{-2}$ by applying \eqref{I2:holder} together with \eqref{eq:norm-two} and \eqref{PNxf}
\begin{align*}
 \sum_{\substack{(N_{0},N_{1},N_2)\in\mathcal{N} \\ N_0\le \langle s\rangle^{-2}}}\normo{\mathcal{I}_{\textbf{N}}^{2}(s)}_{L^{2}}
 &\les |s|\sum_{\substack{(N_{0},N_{1},N_2)\in\mathcal{N} \\ N_0\le \langle s\rangle^{-2}}}\langle N_0\rangle^2\normo{\textbf{m}_{\mathbf{N}}}_{L^{\infty}}\normo{P_{N_2}|u|^{2}(s)}_{L^{2}}\normo{P_{N_{1}}{f}(s)\,}_{L^{2}} \\ 
 &\les |s|\sum_{\substack{(N_{0},N_{1},N_2)\in\mathcal{N} \\ N_0\le \langle s\rangle^{-2}}}\langle N_0\rangle^2N_0\max(\langle N_0\rangle, \langle N_1\rangle)^{-1}  N_2\langle N_2\rangle^{-\frac{k}{2}}N_1\langle N_1\rangle^{-k}\ep_1^3 \\ 
 &\les \langle s\rangle^{-1+\delta_0}\ep_1^3,
\end{align*}
where in the second inequality we used Lemma~\ref{lem:quadratic terms}.
For the remaining contribution, we utilize \eqref{I2:CM} together with \eqref{eq:norm-infty} and \eqref{PNxf} to obtain 
\begin{align*}
    \sum_{\substack{(N_{0},N_{1},N_2)\in\mathcal{N} \\ N_0\ge \langle s\rangle^{-2}}}\normo{\mathcal{I}_{\textbf{N}}^{2}(s)}_{L^{2}}
    &\les |s|\sum_{N_0\ge \langle s\rangle^{-2}} \langle N_0\rangle^2\max(\langle N_0\rangle, \langle N_1\rangle)^{-1} \normo{P_{L}|u|^{2}(s)}_{L^{\infty}}\normo{P_{N_{1}}{f}(s)}_{L^{2}} \\ 
    &\les|s|\sum_{\substack{(N_{0},N_{1},N_2)\in\mathcal{N} \\N_0\ge \langle s\rangle^{-2},\;  N_2\le \langle s\rangle^{-1}}}\langle N_0\rangle^2\max(\langle N_0\rangle, \langle N_1\rangle)^{-1}N_2^2N_1\langle N_1\rangle^{-k}\ep_1^3\\ 
    &\;\;\; + \langle s\rangle^{-1}\sum_{\substack{(N_{0},N_{1},N_2)\in\mathcal{N} \\N_0\ge \langle s\rangle^{-2},\;  N_2\ge \langle s\rangle^{-1}}}\langle N_0\rangle^2\max(\langle N_0\rangle, \langle N_1\rangle)^{-1}\langle N_2\rangle^{-k}N_1\langle N_1\rangle^{-k}\ep_1^3 \\ 
    &\les \langle s\rangle^{-1+\delta_0}\ep_1^3.
\end{align*}
\end{proof}

Let us move on to the proof of second weighted estimates \eqref{eq:second-moment}. 
\begin{proof}[Proof of \eqref{eq:second-moment}]
By Plancherel's theorem, we have
\[
\|x^{2}e^{it\jp D}u\|_{H^2}\sim\|\langle\xi\rangle^2\mathcal{F}(x^{2}e^{it\jp D}u)\|_{L^{2}}\sim\left\Vert \langle\xi\rangle^2\nabla_{\xi}^{2}\widehat{f}\,\right\Vert _{L^{2}}.
\]
The Duhamel's formula \eqref{eq:duhamel} implies that $\nabla_{\xi}^{2}\widehat{f}$
can be represented by 
\begin{align*}
    \langle\xi\rangle^2\nabla_{\xi}^{2}\widehat{f}(t,\xi) & =\langle\xi\rangle^2\nabla^{2}\widehat{u_{0}}(\xi)+\frac{i\lam}{2\pi}\sum_{j=1}^{4}\int_0^t\mathcal{J}^{j}(s,\xi)ds,
\end{align*}
where, by abusing the notation,
\begin{align}
\mathcal{J}^{1}(s,\xi) & =\langle\xi\rangle^2\int_{\mathbb{R}^{2}}e^{is\phi(\xi,\eta)}|\eta|^{-1}\nabla^{2}\widehat{f}(\xi-\eta)\mathcal{F}(|u|^{2})(\eta)d\eta ,\nonumber \\
\mathcal{J}^{2}(s,\xi) & =2is\langle\xi\rangle^2\int_{\mathbb{R}^{2}}\nabla_{\xi}\phi(\xi,\eta)e^{is\phi(\xi,\eta)}|\eta|^{-1}\nabla\widehat{f}(\xi-\eta)\mathcal{F}(|u|^{2})(\eta)d\eta ,\nonumber \\
\mathcal{J}^{3}(s,\xi) & =is\langle\xi\rangle^2\int_{\mathbb{R}^{2}}\nabla_{\xi}^2\phi(\xi,\eta)(\xi,\eta)e^{is\phi(\xi,\eta)}|\eta|^{-1}\widehat{f}(\xi-\eta)\mathcal{F}(|u|^{2})(\eta)d\eta ,\nonumber \\
\mathcal{J}^{4}(s,\xi) & =-s^{2}\langle\xi\rangle^2\int_{\mathbb{R}^{2}}\big( \nabla_{\xi}\phi(\xi,\eta)\big)^{2}e^{is\phi(\xi,\eta)}|\eta|^{-1}\widehat{f}(\xi-\eta)\mathcal{F}(|u|^{2})(\eta)d\eta. \label{eq:j4-esti}
\end{align}
Then we prove that 
\[
\sum_{j=1}^{4}\normo{\mathcal{J}^{j}(s,\xi)}_{L_{\xi}^{2}}\les\langle s\rangle^{-1+2\de_{0}}\ve_{1}^{3}, \quad \text{ for } j=1,\cdots,4.
\]
\noindent \uline{Estimates for \mbox{$\mathcal{J}^{1}$}}. By Lemma \ref{lem:norm-esti},
$\mathcal{J}^1$ can be bounded as in \eqref{I1}. Indeed, 
\begin{align*}
\normo{\langle\xi\rangle^2\mathcal{J}^{1}(s,\xi)}_{L_{\xi}^{2}(\R^2)} 
& \les  \| u(s)\|_{L^\infty} \| u(s)\|_{H^2}\|x^2f(s)\|_{H^2}   \les \bra{s}^{-1+2\de_{0}}\ve_{1}^{3}.
\end{align*}

\noindent \uline{Estimates for \mbox{$\mathcal{J}^{2}$}}.
Estimates for $\mathcal{J}^{2}$ can be done almost similarly to those for $\mathcal{I}^{2}$. 
If one follows the argument, the only difference is that the norm $\|P_{N_1}f(s)\|_{L^2}$ in inequalities \eqref{I2:holder} and \eqref{I2:CM} is replaced by the weighted norm $\|P_{N_1}xf(s)\|_{L^2}$, which can be easily dealt with once one utilizes \eqref{PNx2f}.

% Here, we encounter $\|P_{N_1}xf(s)\|_{L^2}$
% in \eqref{I2:holder} and \eqref{I2:CM} instead of $\|P_{N_1}f(s)\|_{L^2}$. 
% But, the desired bound can be obtained if one utilizes \eqref{PNxf}
% \begin{align*}
%     \|P_{N}xe^{it\langle D\rangle}u(t)\|_{L^2} \les \min( N^{\delta_0} \langle t\rangle^{\frac32\delta_0}, \langle N\rangle^{-2}\langle t\rangle^{\delta_0})\ep_1.
% \end{align*}

\medskip

\noindent \uline{Estimates for \mbox{$\mathcal{J}^{3}$}}.
% Estimates for $\mathcal{J}^{1}$ and $\mathcal{J}^{2}$ can be done almost similarly to those for $\mathcal{I}^{1}$ and $\mathcal{I}^{2}$, respectively. The only
% difference between them is the order of derivatives which falls on $\wh{f}$ and the $L^2$ norm of derivative of $f$ up to the second order is bounded in our function space. Thus, we
% omit the proof of estimates for $\mathcal{J}^{1}$ and $\mathcal{J}^{2}$. In addition, 
$\mathcal{J}^{3}$ also can be handled in a similar manner to $\mathcal{I}^{2}$.  
Indeed, one finds that the multiplier $\nabla_{\xi}^2\phi(\xi,\eta)$ in $\mathcal{J}^{3}$ verifies the smaller bound than the one given in \eqref{bound of grad phi} satisfied by the multiplier $\nabla_{\xi}\phi(\xi,\eta)$ in $\mathcal{I}^{2}$. 
More precisely, the following bound holds
\begin{align*}
    \left| \nabla_{\xi}^2\phi(\xi,\eta)\right| 
    \les  \frac{|\eta|}{\max(\langle\xi\rangle,\langle\xi-\eta\rangle)^2}.
\end{align*}
Thus, if we let 
\begin{align*}
    \widetilde{\mathbf{m}}_{\mathbf{N}}(\xi,\eta) 
    =\langle \xi \rangle^2  |\eta|^{-1}  \nabla_{\xi}^2\phi(\xi,\eta)\rho_{N_{0}}(\xi)\rho_{N_2}(\eta)\rho_{N_{1}}(\xi-\eta) ,
\end{align*}
one can show that 
\begin{align*}
    \sup_{\xi,\eta\in\R^2}\left|\widetilde{\mathbf{m}}_{\mathbf{N}}(\xi,\eta)\right|&\les \langle N_0\rangle^2 \max(\langle N_0\rangle, \langle N_1\rangle)^{-2}, \\ C_{\widetilde{\mathbf{m}}_{\mathbf{N}}} &\les  \langle N_0\rangle^2\max(\langle N_0\rangle, \langle N_1\rangle)^{-2}.
\end{align*} 
Applying these bounds into \eqref{I2:holder} and \eqref{I2:CM}, one can obtain the desired bounds.

\medskip

\noindent \uline{Estimates for \mbox{$\mathcal{J}^{4}$}}.
It remains to estimate the main case $\mathcal{J}^{4}$. 
As before, we decompose 
\begin{align*}
    \mathcal{J}^{4 }(s,\xi)&=\sum_{\textbf{N}=(N_{0},N_{1},N_2)\in (2^{\Z})^3}\mathcal{J}_{\textbf{N}}^{4}(s,\xi),\\
 \mathcal{J}_{\textbf{N}}^{4}(s,\xi) &:=-s^2\int_{\mathbb{R}^{2}}\mathbf{m}_{\mathbf{N}}(\xi,\eta)e^{is\phi(\xi,\eta)}\widehat{P_{N_1}f}(\xi-\eta)\widehat{P_{N_2}(|u|^{2})}(\eta)\,d\eta,
\end{align*}
where 
\begin{align*}
    \mathbf{m}_{\mathbf{N}}(\xi,\eta) 
=\langle \xi\rangle^2   |\eta|^{-1} \left( \nabla_{\xi}\phi(\xi,\eta)\right)^2\rho_{N_{0}}(\xi)\rho_{N_2}(\eta)\rho_{N_{1}}(\xi-\eta).
\end{align*}
One can readily verify that  
\begin{align}\begin{aligned}\label{pointwise bound of m2}
    \sup_{\xi,\eta\in\R^2}\left|\mathbf{m}_{\mathbf{N}}(\xi,\eta)\right|&\les N_2\langle N_0\rangle^2\max(\langle N_0\rangle, \langle N_1\rangle)^{-2}, \\ 
     C_{\mathbf{m}_{\mathbf{N}}}&\les  N_2\langle N_0\rangle^2\max(\langle N_0\rangle, \langle N_1\rangle)^{-2}.
\end{aligned}\end{align} 
We see that the sum over those indexes $N_0$ such that $N_0\le \langle s\rangle^{-3}$ can be dealt with:
\begin{align*}
    \|\mathcal{J}_{\textbf{N}}^{4}(s)\|_{L^2(\R^2)} &\les |s|^2\sum_{\substack{ (N_0,N_1,L)\in \mathcal{N} \\ N_0 \le \langle s\rangle^{-3}}}\|\rho_{N_{0}}\|_{L^{2}}\normo{\textbf{m}_{\mathbf{N}}}_{L^{\infty}}\normo{P_{N_2}|u|^{2}(s)}_{L^{2}}\normo{P_{N_{1}}{f}(s)\,}_{L^{2}} \\ 
    &\les   |s|^2   \sum_{N_0 \le \langle s\rangle^{-3}}N_0N_2\langle N_0\rangle^2\max(\langle N_0\rangle, \langle N_1\rangle)^{-2}N_2\langle N_2\rangle^{-\frac{k}{2}}N_1\langle N_1\rangle^{-k}\ep_1^3\\ 
    &\les \langle s\rangle^{-1}\ep_1^3.
\end{align*}
On the other hand, by the multiplier inequalities with \eqref{pointwise bound of m2},
 one has 
\begin{align}\label{J4}
    \|\mathcal{J}_{\textbf{N}}^{4}(s)\|_{L^2(\R^2)} 
    \les   |s|^2   \sum_{\substack{ (N_0,N_1,L)\in \mathcal{N} \\ N_0 \ge \langle s\rangle^{-3}}}N_2\langle N_0\rangle^2\max(\langle N_0\rangle, \langle N_1\rangle)^{-2}\normo{P_{N_2}|u(s)|^{2}}_{L^{2}}\normo{P_{N_{1}}{u}(s)\,}_{L^{\infty}}.
\end{align}
Using \eqref{eq:norm-two}, we can bound the sum in \eqref{J4} for $N_2 \le \langle s\rangle^{-1+\frac34\delta_0}$ by 
\begin{align*}
    &|s|^2   \sum_{ \substack{ N_0 \ge \langle s\rangle^{-3}, \\  N_2 \le \langle s\rangle^{-1+\frac34\delta_0} }}N_2\langle N_0\rangle^2\max(\langle N_0\rangle, \langle N_1\rangle)^{-2} N_2
     N_1^{\frac14\delta_0}\langle N_1\rangle^{-k-\frac14\delta_0}\langle t\rangle^{-1+\frac14\delta_0}\ep_1^3  \\ 
     &\les \langle s\rangle^{-1+\frac74\delta_0} 
     \sum_{N_0\ge \langle s\rangle^{-3},\; N_0\les N_1 } \langle N_0\rangle^2\max(\langle N_0\rangle, \langle N_1\rangle)^{-2} 
     N_1^{\frac14\delta_0}\langle N_1\rangle^{-k-\frac14\delta_0}\ep_1^3 \\
     &\qquad +\langle s\rangle^{-1+\frac74\delta_0} \sum_{N_1\ll N_0 \les \langle s\rangle^{-1+\frac34\delta_0}} 
     N_1^{\frac14\delta_0} \ep_1^3 \\ 
     &\les  \langle s\rangle^{-1+2\delta_0}\ep_1^3.
   \end{align*}
To estimate the sum in \eqref{J4} for $L \ge \langle s\rangle^{-1+\frac34\delta_0}$, we use \eqref{ineq:PLu2} to obtain 
\begin{align*}
    &\langle s\rangle^{-1+\frac74\delta_0}   \sum_{ \substack{ N_0 \ge \langle s\rangle^{-3}, \\ N_2 \ge \langle s\rangle^{-1+\frac34\delta_0} }}\langle N_0\rangle^2\max(\langle N_0\rangle, \langle N_1\rangle)^{-2} \langle N_2\rangle^{-1}
    N_1^{\frac14\delta_0}\langle N_1\rangle^{-k-\frac14\delta_0}\ep_1^3  \\ 
    &\les \langle s\rangle^{-1+2\delta_0}\ep_1^3.
\end{align*}
\end{proof}

\section{Modified Scattering : Proof of Theorem \ref{main-thm:semi}}\label{sec:scattering}

\global\long\def\freq{{(\xi,\eta,\sigma)}}%
In Section \ref{sec:Weighted-Energy-estimate}, we have shown that the small solutions stay small in weighted Sobolev norm. In this section we prove \eqref{eq:modified-scattering}, the asymptotic behaviors of solutions, 
and complete the proof of Theorem~\ref{main-thm:semi}. We assume that $u$ satisfies
the a priori assumption \eqref{assumption-apriori}.
The modified scattering profile is defined by 
\begin{align*}
    \mathsf{v}(t,\xi)=e^{-iB(t,\xi)}e^{it\braxi}\wh{u}(t,\xi),
\end{align*}
where the phase correction is given by 
\begin{align*}
B(t,\xi) & =\frac{\lam}{(2\pi)^{2}}\int_{0}^{t}\int_{\mathbb{R}^{2}}\left|\frac{\xi}{\langle\xi\rangle}-\frac{\sigma}{\langle\sigma\rangle}\right|^{-1}\abs{\wh{u}(\sigma)}^{2}d\sigma\frac{\rho(s^{-\frac{2}{n}}\xi)}{\bra{s}}ds.
\end{align*}
\begin{prop}\label{prop:scattering} 
Assume that $u\in C([0,T],H^{n})$
satisfies the a priori assumption \eqref{assumption-apriori} with the index conditions \eqref{eq:bundle}. Then we
get 
\begin{align}
    \normo{\braxi^{k}\Big(\mathsf{v}(t_{2},\xi)-\mathsf{v}(t_{1},\xi)\Big)}_{L_{\xi}^{\infty}}\les\ve_{1}^3\bra{t_{1}}^{-\de}.\label{eq:scattering}
\end{align}
for $t_{1}\le t_{2}\in[0,T]$ and some $0<\de\le\frac{1}{100}$. \end{prop} 

By assuming Proposition~\ref{prop:scattering}, we first prove our main theorem.

\smallskip

\begin{proof}[Proof of Theorem \ref{main-thm:semi}]
For the proof of global behavior of solutions, it takes precedence to show the existence of a local solution to \eqref{main-eq:semi} in $\Sigma_{T}$. However,
since it is straightforward from the contraction mapping principle, we may omit the proof (for instance see
\cite{choz2006-siam,lee2021-bkms,hele2014}).
Now, given $T>0$, we assume that $\psi$ is a solution to \eqref{main-eq:semi}
on $[0,T]$ with initial data condition \eqref{condition-initial:semi}.
Then, by the bootstrap argument, it suffices to prove that for sufficiently small $\ve_{1}>0$,
there exists $C>0$ such that if $\|u\|_{\Sigma_{T}}\le \ep_1$,
\begin{align}
\|u\|_{\Sigma_{T}}\le\ve_{0}+C\ve_{1}^{3}.\label{eq:contraction}
\end{align}
From \eqref{eq:scattering}, one sees that the scattering norm stays bounded
\[
\|u(t)\|_{S}\le\ve_{0}+C\ve_{1}^{3}, \;\;\text{ for } t\in[0,T].
\]
Thus, together with the weighted energy estimates \eqref{eq:first-moment}
and \eqref{eq:second-moment}, we can close the bootstrapping argument and obtain the global existence of solution.
Concerning \eqref{eq:modified-scattering}, the asymptotic behavior of the solution,
we define a scattering profile by 
\[
u_{\infty}:=\mathcal{F}^{-1}\left(\lim_{t\to\infty}\mathsf{v}(t,\cdot)\right).
\]
Then Proposition \ref{prop:scattering} immediately yields that for $t\in[0,T]$,
\begin{align*}
\left\Vert \bra{\xi}^{k}\left(\widehat{u}(t,\xi)-e^{iB(t,\xi)}e^{-it\bra{\xi}}\wh{u_{\infty}}(\xi)\right)\right\Vert _{L_{\xi}^{\infty}}\les\ve_{1}^3\bra{t}^{-\de}.
\end{align*}    
\end{proof}

\smallskip 

\noindent\emph{Proof of Proposition \ref{prop:scattering}.}
We will proceed as in \cite{pusa}. We prove that if $t_1\le t_{2}\in [M-2,2M]\cap[0,T]$ for
a dyadic number $M\in2^{\mathbb{N}}$
\begin{align}
\normo{\braxi^{k}\Big(\mathsf{v}(t_{2},\xi)-\mathsf{v}(t_{1},\xi)\Big)}_{L_{\xi}^{\infty}}\les\ve_{1}^{3}M^{-\de},\label{goal-modi}
\end{align}
for some $0<\delta\le\frac{1}{100}$.
We begin with writing $\mathsf{v}$ as 
\[
\mathsf{v}(t,\xi)=e^{-iB(t,\xi)}\wh{f}(t,\xi)=e^{-iB(t,\xi)}\wh{u_{0}}(\xi)+i\lam e^{-iB(t,\xi)}\mathcal{I}(t,\xi),
\]
where, after change of variables, the nonlinear term $\mathcal{I}$ is given by 
\begin{align*}
 & \mathcal{I}(t,\xi)=\frac{1}{(2\pi)^3}\int_{0}^{t}\iint_{\R^{2}\times\R^{2}}e^{isq\freq}|\eta|^{-1}\wh{f}(s,\xi+\eta)\wh{f}(s,\xi+\sigma)\overline{\wh{f}(s,\xi+\eta+\sigma)}d\eta d\sigma ds,
\end{align*}
% where we applied 
% The nonlinear part comes from \eqref{eq:duhamel} by the change of
% variables with 
% \[
% \sigma\to\xi-\eta+\sigma\;\;\mbox{ and }\;\;\eta\to-\eta.
% \]
% In the view of \eqref{function-resonance}, 
where a resonant function
\[
q\freq=\braxi-\langle\xi+\eta\rangle-\langle\xi+\sigma\rangle+\bra{\xi+\eta+\sigma}.
\]
Let $L_{0}\in2^{\Z}$ such that 
\begin{align}
L_{0}\sim M^{-\frac{9}{10}}.\label{eq:lzero}
\end{align}
We write  
\begin{align}
\begin{aligned}\label{eq:duhamel-lzero}\mathcal{I}(t,\xi)=\int_{0}^{t}\bigg(\mathcal{K}_{L_{0}}(s,\xi)+\sum_{L\in2^{\Z},L>L_{0}}\mathcal{K}_{L}(s,\xi)\bigg)ds,\end{aligned}
\end{align}
where 
\begin{align*}
\mathcal{K}_{L_{0}}(s,\xi) & :=\frac{1}{(2\pi)^3}\iint_{\R^{2}\times\R^{2}}e^{isq\freq}\rho_{\le L_{0}}(\eta)|\eta|^{-1}\wh{f}(s,\xi+\eta)\wh{f}(s,\xi+\sigma)\overline{\wh{f}(s,\xi+\eta+\sigma)}d\eta d\sigma,\\
\mathcal{K}_{L}(s,\xi) & :=\frac{1}{(2\pi)^3}\iint_{\R^{2}\times\R^{2}}e^{isq\freq}\rho_{L}(\eta)|\eta|^{-1}\wh{f}(s,\xi+\eta)\wh{f}(s,\xi+\sigma)\overline{\wh{f}(s,\xi+\eta+\sigma)}d\eta d\sigma.
\end{align*}
The first term $\mathcal{K}_{L_{0}}$, the integral around the singular point,
is the one responsible for the correction of scattering, whereas the second term $\mathcal{K}_{L}$ is remainder term. The profile $\mathsf{v}$ verifies  
\begin{align*}
\partial_{t}\mathsf{v}(t,\xi) & =\partial_{t}\left[e^{-iB(t,\xi)}\wh{f}(t,\xi)\right]\\
 & =i\lam e^{-iB(t,\xi)}\left[\left(\mathcal{K}_{L_{0}}(t,\xi)+\sum_{L>L_{0}}\mathcal{K}_{L}(t,\xi)\right)-\frac{1}{\lam}\left[\partial_{t}B(t,\xi)\right]\wh{f}(t,\xi)\right].
\end{align*}
Thus,
\begin{align}
\begin{aligned}\label{eq:deriv-v}  \mathsf{v}(t_{2},\xi)-\mathsf{v}(t_{1},\xi)
  &=\int_{t_{1}}^{t_{2}}\partial_{s}\mathsf{v}(s,\xi)ds\\
 & =i\lam\int_{t_{1}}^{t_{2}}e^{-iB(s,\xi)}\left[\left(\mathcal{K}_{L_{0}}(s,\xi)+\sum_{L>L_{0}}\mathcal{K}_{L}(s,\xi)\right)-\frac{1}{\lam}\left[\partial_{s}B(s,\xi)\right]\wh{f}(s,\xi)\right]ds.
\end{aligned}
\end{align}
In order to prove \eqref{goal-modi}, we use the cancellation effect
between $ \mathcal{K}_{L_{0}} $ and $\partial_{s}B(s,\xi)$, specifically, we show that for each $\xi$ with $|\xi|\sim N\in2^{\Z}$,
\begin{align}
\begin{aligned} & \left|\int_{t_{1}}^{t_{2}}e^{-iB(s,\xi)}\left(\mathcal{K}_{L_{0}}(s,\xi)-\frac{1}{\lam}\left[\partial_{s}B(s,\xi)\right]\wh{f}(s,\xi)\right)ds\right|\les\ve_{1}^{3}M^{-\de}\bra{N}^{-k}\label{eq:part-modification}\end{aligned}
\end{align}
and 
\begin{align}
\left|\int_{t_{1}}^{t_{2}}e^{-iB(s,\xi)}\sum_{L>L_{0}}\mathcal{K}_{L}(s,\xi)ds\right| & \les\ve_{1}^{3}M^{-\de}\bra{N}^{-k},\label{eq:part-scattering}
\end{align}
for some $0<\delta\le \frac{1}{100}$.

\smallskip
\noindent\emph{Proof of \eqref{eq:part-modification}.} 
% In the estimate \eqref{eq:part-modification}, when $M^{\frac{3}{n}}\les N$,
% we can get the desired decay \eqref{goal-modi} without the phase
% modification $B(s,\xi)$. Then we progress our proof to separate into
% $M^{\frac{3}{n}}\les N$ and $N\ll M^{\frac{3}{n}}$ via $\rho\in C_{0}^{\infty}(B(0,2))$.
We prove the following two bounds:
\begin{align}
\left|\int_{t_{1}}^{t_{2}}e^{-iB(s,\xi)}\mathcal{K}_{L_{0}}(s,\xi)\left(1-\rho\left(s^{-\frac{2}{n}}\xi\right)\right)ds\right| & \les\ve_{1}^{3}M^{-\de}\bra{N}^{-k},\label{eq:decay-part}\\
\left|\int_{t_{1}}^{t_{2}}e^{-iB(s,\xi)}\left[\mathcal{K}_{L_{0}}(s,\xi)\rho\left(s^{-\frac{2}{n}}\xi\right)-\frac{1}{\lam}\left[\partial_{s}B(s,\xi)\right]\wh{f(s,\xi)}\right]ds\right| & \les\ve_{1}^{3}M^{-\de}\bra{N}^{-k},\label{eq:crucial-part}
\end{align}
where $\rho\in C_{0}^{\infty}(B(0,2))$.
We remark that the phase correction will be derived in the proof of \eqref{eq:crucial-part}.

We first consider \eqref{eq:decay-part}. It suffices to show the integrand bound 
\begin{align*}
\Big|\mathcal{K}_{L_{0}}(s,\xi)\Big|\les\ve_{1}^{3}M^{-(1+\de)}\bra{N}^{-k}.
\end{align*}
%By a priori assumption \eqref{assumption-apriori}, we estimate
%\begin{align*}
%	& \left|\mathcal K_{L_{0}}(s,\xi)\right|\\
%	& \;\;\les \bra{\xi}^{-k}\int\!\!\!\! \int_{\R^2 \times \R^2}\left|\bra{N}^k\rho_{\le L_{0}}(\eta)|\eta|^{-1}\wh{f}(s,\xi+\eta)\wh{f}(s,\xi+\sigma)\overline{\wh{f}(s,\xi+\eta+\sigma)}\right|d\eta d\sigma\\
%	& \;\;\les L_{0}\bra{N}^{-k}\left\|\bra{\xi}^k\wh{f}\,\right\|_{L_{\xi}^{\infty}}\|u\|_{H^{k}}\|u\|_{H^k}\\
%	& \;\;\les \ve_{1}^{3} L_{0}M^{2\de_{0}}\bra{N}^{-k}.
%\end{align*}
%Then, \eqref{eq:lzero} implies our desired result.
We further split $\mathcal{K}_{L_{0}}$ dyadically as follows:
\begin{align*}
\mathcal{K}_{L_{0}}(s,\xi)=\sum_{L_{1}\le L_{0}+10}\mathcal{K}_{L_{0},L_{1}}(s,\xi),
\end{align*}
where 
\begin{align*}
\mathcal{K}_{L_{0},L_{1}}(s,\xi)  =\frac{1}{(2\pi)^3}\int\!\!\!\!\int e^{isq\freq}\rho_{L_{1}}(\eta)\rho_{\le L_{0}}(\eta)|\eta|^{-1}\wh{f}(s,\xi+\eta)\wh{f}(s,\xi+\sigma)\overline{\wh{f}(s,\xi+\eta+\sigma)}d\eta d\sigma.
\end{align*}
Then, by the a priori assumption \eqref{assumption-apriori}, we estimate
\begin{align*}
 & \left|\mathcal{K}_{L_{0},L_{1}}(s,\xi)\right|\\
 & \;\;\les L_{1}^{-1}\bra{\xi}^{-k}\int\!\!\!\!\int\left|\bra{N}^{k}\rho_{L_{1}}(\eta)\rho_{\le L_{0}}(\eta)\wh{f}(s,\xi+\eta)\wh{f}(s,\xi+\sigma)\overline{\wh{f}(s,\xi+\eta+\sigma)}\right|d\eta d\sigma\\
 & \;\;\les L_{1}^{-1}\bra{N}^{-k}L_{1}^{2}\left\Vert \bra{\xi}^{k}\wh{f}\,\right\Vert _{L_{\xi}^{\infty}}\|u\|_{H^{n}}\|u\|_{H^{n}}\\
 & \;\;\les\ve_{1}^{3}L_{1}M^{2\delta_0}\bra{N}^{-k}.
\end{align*}
On the other hand, by H\"older inequality, we get 
\begin{align*}
\left|\mathcal{K}_{L_{0},L_{1}}(s,\xi)\right|\;\;\les L_{1}^{-1}\|\rho_{L_{1}}\|_{L^{2}}N^{-n}\|f\|_{H^{n}}^{3}\les\ve_{1}^{3}M^{-2}M^{3\delta_0}.
\end{align*}
These two estimates induce that 
\begin{align*}
 & \sum_{L_{1}\le L_{0}+10}\Big|\mathcal{K}_{L_{0},L_{1}}(s,\xi)\Big|\\
 & \les\ve_{1}^{3}\left(\sum_{L_{1}\le M^{-2}}L_{1}M^{2\delta_0}\bra{N}^{-k}+\sum_{M^{-2}< L_{1}<L_0+10}M^{-2+3\delta_0}\right)\\
 & \les\ve_{1}^{3}M^{-(1+\delta_0)}\bra{N}^{-k}.
\end{align*}

Next, consider \eqref{eq:crucial-part}. 
Due to the cut-off $\rho(s^{-\frac{2}{n}}\xi)$, we may assume $N\le M^{\frac{2}{n}}$.
It suffices to show that the integrand satisfies 
\begin{align}
\left|\mathcal{K}_{L_{0}}(s,\xi)-\frac{1}{\lam}\left[\partial_{s}B(s,\xi)\right]\wh{f}(s,\xi)\right| & \les\ve_{1}^{3}M^{-(1+\de)}\bra{N}^{-k}.\label{eq:proof-part2}
\end{align}
The correction term will be achieved after  three steps. 

\noindent\textit{Step1: Phase approximation.}
We approximate the phase function by a simpler one in the support of the integrand in \eqref{eq:crucial-part}. Let us observe that 
\begin{align*}
q\freq & =\Big(\braxi-\bra{\xi+\eta}\Big)-\Big(\bra{\xi+\sigma}-\bra{\xi+\eta+\sigma}\Big)\\
 & =\left(\frac{-|\eta|^{2}-2\eta\cdot\xi}{\braxi+\bra{\xi+\eta}}-\frac{-|\eta|^{2}-2\eta\cdot(\xi+\sigma)}{\bra{\xi+\sigma}+\bra{\xi+\eta+\sigma}}\right)\\
 & =\eta\cdot\left(\frac{\xi}{\braxi}-\frac{\xi+\sigma}{\bra{\xi+\sigma}}\right)+O\left(|\eta|^{2}\right)\\
 & =:r(\xi,\eta,\sigma)+O\left(|\eta|^{2}\right).
\end{align*}
We now set 
\begin{align*}
\mathcal{K}_{L_{0}}'(s,\xi) & :=\frac{1}{(2\pi)^3}\iint_{\R^{2}\times\R^{2}}e^{isr\freq}\rho_{\le L_{0}}(\eta)|\eta|^{-1}\wh{f}(s,\xi+\eta)\wh{f}(s,\xi+\sigma)\overline{\wh{f}(s,\xi+\eta+\sigma)}d\eta d\sigma.
\end{align*}
Then we estimate 
\begin{align*}
 & \left|\mathcal{K}_{L_{0}}(s,\xi)-\mathcal{K}_{L_{0}}'(s,\xi)\right|\\
 & \les\iint_{\R^{2}\times\R^{2}}|s|\left|q(\xi,\eta,\sigma)-r(\xi,\eta,\sigma)\right||\eta|^{-1}\\
 & \hspace{2.2cm}\times\left|\rho_{\le L_{0}}(\eta)\wh{f}(s,\xi+\eta)\wh{f}(s,\xi+\sigma)\overline{\wh{f}(s,\xi+\eta+\sigma)}\right|d\eta d\sigma\\
 & \les M\iint_{\R^{2}\times\R^{2}}|\eta|\left|\rho_{\le L_{0}}(\eta)\wh{f}(s,\xi+\eta)\wh{f}(s,\xi+\sigma)\overline{\wh{f}(s,\xi+\eta+\sigma)}\right|d\eta d\sigma\\
 & \les ML_{0}^{3}\normo{f}_{L^{2}}^{2}\norm{\wh{f}}_{L_{\xi}^{\infty}}\\
 & \les\ve_{1}^{3}M^{-\frac{17}{10}}\les\ve_{1}^{3}M^{-(1+\de)}\bra{N}^{-k},
\end{align*}
where we used \eqref{eq:lzero} and $\bra{N}^{k} \le M^{\frac{2k}{n}}\le M^{\frac{1}{100}}$ in the last inequality. 

\noindent\textit{Step2: Profiles approximation.}
We now approximate
$\mathcal{K}_{L_{0}}'$ by $\wt{\mathcal{K}_{L_{0}}\;}$
which is defined by 
\begin{align*}
\wt{\mathcal{K}_{L_{0}}\;}(s,\xi) & :=\frac{1}{(2\pi)^3}\iint_{\R^{2}\times\R^{2}}e^{isr\freq}\rho_{\le L_{0}}(\eta)|\eta|^{-1}\wh{f}(\xi)\left|\wh{f}(\xi+\sigma)\right|^{2}d\eta d\sigma.
\end{align*}
By setting $R=L_{0}^{-\frac{1}{2}}$, we see that 
\begin{align*}
 \left|\wh{f}(\zeta+\eta)-\wh{f}(\zeta)\right|
 & \les\left|\wh{\rho_{>R}f}(\zeta+\eta)-\wh{\rho_{>R}f}(\zeta)\right|+\left|\wh{\rho_{\le R}f}(\zeta+\eta)-\wh{\rho_{\le R}f}(\zeta)\right|\\
 & \les\normo{\wh{\rho_{>R}f}}_{L_{\xi}^{\infty}}+L_{0}\normo{\nabla_{\xi}\wh{\rho_{\le R}f}}_{L_{\xi}^{\infty}}\\
 & \les R^{-1}\normo{\bra{x}^{2}f}_{L^{2}}+L_{0}\|\rho_{\le R}\|_{L_{x}^{2}}\normo{xf}_{L^{2}}\\
 & \les L_{0}^{\frac{1}{2}}M^{2\de_{0}}.
\end{align*}
From this and \eqref{eq:lzero}, we estimate 
\begin{align*}
 & \abs{\mathcal{K}_{L_{0}}'(s,\xi)-\wt{\mathcal{K}_{L_{0}}\;}(s,\xi)}\\
 & \les\iint_{\R^{2}\times\R^{2}}\rho_{\le L_{0}}(\eta)|\eta|^{-1}\abs{\wh{f}(\xi+\eta)\wh{f}(\xi+\sigma)\overline{\wh{f}(\xi+\eta+\sigma)}-\wh{f}(\xi)\left|\wh{f}(\xi+\sigma)\right|^{2}}d\eta d\sigma\\
 & \les\ve_{1}^{3}L_{0}^{\frac{3}{2}}M^{2\de_{0}}\les\ve_{1}^{3}M^{-(1+\de)}\bra{N}^{-k}.
\end{align*}

\noindent\textit{Step3: final approximation.}
We conclude the proof of \eqref{eq:crucial-part} by showing that 
\begin{align}
\left|\wt{\mathcal{K}_{L_{0}}\;}(s,\xi)-\frac{1}{\lam}\left[\partial_{s}B(s,\xi)\right]\wh{f}(s,\xi)\right|\les\ve_{1}^{3}M^{-(1+\de)}\bra{N}^{-k}.\label{eq:cancel}
\end{align}
Setting $\mathbf{z}:=\left(\frac{\xi}{\bra{\xi}}-\frac{\sigma}{\bra{\sigma}}\right)$,
by the change of variables, we get 
\begin{align*}
\wt{\mathcal{K}_{L_{0}}\;}(s,\xi) & =\frac{1}{(2\pi)^3}\iint_{\R^{2}\times\R^{2}}e^{is\eta\cdot\mathbf{z}}\rho_{\le L_{0}}(\eta)|\eta|^{-1}\wh{f}(\xi)\left|\wh{f}(\sigma)\right|^{2}d\eta d\sigma.
\end{align*}
Then, \eqref{eq:cancel} can be reduced to showing that 
\begin{align}\begin{aligned}\label{eq:cancel2}
    &\left|\frac{1}{(2\pi)^3}\iint_{\R^{2}\times\R^{2}}e^{is\eta\cdot\mathbf{z}}\rho_{\le L_{0}}(\eta)|\eta|^{-1}\abs{\wh{f}(\sigma)}^{2}d\eta d\sigma-\frac{1}{(2\pi)^{2}|s|}\int_{\R^{3}}|\mathbf{z}|^{-1}\abs{\wh{f}(\sigma)}^{2}d\sigma\right| \\
    &\quad \les \ep_1^2 M^{-(1+\de)}.
\end{aligned}\end{align}
% Then, by the definition of $B(s,\xi)$, we deduce that 
% \begin{align*}
%  & \left|\wt{\mathcal{K}_{L_{0}}^{q}\;}(s,\xi)-\frac{1}{\lam}\left[\partial_{s}B(s,\xi)\right]\wh{f}(s,\xi)\right|\\
%  & \les\abs{\wh{f}(\xi)}\left|\iint_{\R^{2}\times\R^{2}}e^{is\eta\cdot\mathbf{z}}\rho_{\le L_{0}}(\eta)|\eta|^{-1}\abs{\wh{f}(\sigma)}^{2}d\eta d\sigma-\frac{(2\pi)^{2}}{s}\int_{\R^{3}}|\mathbf{z}|^{-1}\abs{\wh{f}(\sigma)}^{2}d\sigma\right|\\
%  & \les\abs{\wh{f}(\xi)}\left|\int_{\R^{2}}\left|\int_{\R^{2}}e^{is\eta\cdot\mathbf{z}}|\eta|^{-1}\rho_{\le L_{0}}(\eta)d\eta-(2\pi)^{2}|s\mathbf{z}|^{-1}\right|\abs{\wh{f}(\sigma)}^{2}d\sigma\right|.
% \end{align*}
We first claim that 
\begin{align}\label{bound1}
    \left|\int_{\R^{2}}e^{is\eta\cdot\mathbf{z}}|\eta|^{-1}\rho_{\le L_{0}}(\eta)d\eta-(2\pi)^{2}|s\mathbf{z}|^{-1}\right|
    \les   M^{-\frac{16}{15}}|\mathbf{z}|^{-\frac53}.
\end{align}
Observe that since $\mathcal{F}(|x|^{-1})=2\pi|\eta|^{-1}$, the following formula hols 
\[
\frac{2\pi}{ |s\mathbf{z}|}=\lim_{A\to\infty}\int_{\R^2} e^{is\mathbf{z}\cdot\eta}\rho_{\le A}(\eta)\frac{1}{|\eta|}d\eta.
\]
Then we get that for $L_{0}\ll A$, 
\begin{align}
\begin{aligned}\label{eq:final-approx-1} & \left|\int_{\R^{2}}e^{is\eta\cdot\mathbf{z}}|\eta|^{-1}\rho_{\le L_{0}}(\eta)d\eta-2\pi|s\mathbf{z}|^{-1}\right|\\
 & =\left|\int_{\R^{2}}e^{is\eta\cdot\mathbf{z}}|\eta|^{-1}\left(\rho_{\le L_{0}}(\eta)-\rho_{\le A}(\eta)\right)d\eta\right|\\
 & =|s\mathbf{z}|^{-2}\left|\int_{\R^{2}}\left(\nabla_{\eta}^{2}e^{is\eta\cdot\mathbf{z}}\right)|\eta|^{-1}\left(\rho_{\le L_{0}}(\eta)-\rho_{\le A}(\eta)\right)d\eta\right| \\ 
 & \les M^{-2}|\mathbf{z}|^{-2}L_{0}^{-1}
\end{aligned}
\end{align}
and the trivial bounds 
\begin{align}
\begin{aligned}\label{eq:final-approx-2}\left|\int_{\R^{2}}e^{is\eta\cdot\mathbf{z}}|\eta|^{-1}\rho_{\le L_{0}}(\eta)d\eta-(2\pi)^{2}|s\mathbf{z}|^{-1}\right| & \les L_{0}+|s\mathbf{z}|^{-1}.
\end{aligned}
\end{align}
Then, \eqref{bound1} follows by interpolating \eqref{eq:final-approx-1} and \eqref{eq:final-approx-2}, since $L_0\sim M^{-\frac{9}{10}}$.
Now, the left-hand side of \eqref{eq:cancel2} is bounded by 
\begin{align*}
    M^{-\frac{16}{15}} \left|\int_{\R^{2}}|\mathbf{z}|^{-\frac{5}{3}}\abs{\wh{f}(\sigma)}^{2}d\sigma \right|.
\end{align*}
Since $|\mathbf{z}|\gtrsim\min\left\{ 1,|\sigma|,\frac{|\xi-\sigma|}{\bra{\sigma}^{3}}\right\} $,
the a priori assumption \eqref{assumption-apriori} yields that 
\begin{align*}
 \left|\int_{\R^{2}}|\mathbf{z}|^{-\frac{5}{3}}\abs{\wh{f}(\sigma)}^{2}d\sigma\right|\les\ve_{1}^{2}.    
\end{align*}
which completes the proof of \eqref{eq:cancel2}. 

\medskip

\noindent\emph{Proof of \eqref{eq:part-scattering}.} We further localize the frequencies as follows: 
\begin{align*}
\mathcal{K}_{L}(s,\xi) & =\frac{1}{(2\pi)^3}\sum_{\textbf{N}=(N_{1},N_{2},N_{3})\in(2^{\Z})^{3}}\mathcal{K}_{L,\mathbf{N}}(s,\xi),\\
\mathcal{K}_{L,\textbf{N}}(s,\xi) & :=\iint_{\R^{2}\times\R^{2}}e^{isq\freq}\rho_{L}(\eta)|\eta|^{-1}\wh{f_{N_{1}}}(s,\xi+\eta)\wh{f_{N_{2}}}(s,\xi+\sigma)\overline{\wh{f_{N_{3}}}(s,\xi+\eta+\sigma)}d\eta d\sigma.
\end{align*}
We prove that 
\begin{align}
\sum_{L>L_{0},\textbf{N}=(N_1,N_2,N_3)}\left|\mathcal{K}_{L,\textbf{N}}(s,\xi)\right|\les\ve_{1}^{3}M^{-(1+\de)}\bra{N}^{-k}.\label{eq:goal-high}
\end{align}
By H\"older inequality, we readily have 
\[
\left|\mathcal{K}_{L,\mathbf{N}}(s,\xi)\right|\les\prod_{j=1}^{3}\normo{\wh{f_{N_{j}}}(s)}_{L^{2}}.
\]
Using the a priori assumption \eqref{assumption-apriori}, we know that 
\[
\normo{\wh{f_{N_{j}}}(s)}_{L^{2}}\les\min{\left(N_{j}\bra{N_{j}}^{-k},\bra{N_{j}}^{-n}M^{\de_{0}}\right)}\ve_{1}\;\;\mbox{ for }\;\;j=1,2,3.
\]
The last two estimates above imply the summation in \eqref{eq:goal-high} over those indexes $(N_1,N_2,N_3)$ with 
$\max( N_{1},N_{2},N_{3})\ge M^{\frac{2}{n}}$ or 
$\min( N_{1},N_{2},N_{3})\le M^{-(1+\de)}$ satisfy the desired bound.
Thus, we suffice to estimate the sum over those indexes $(N_1,N_2,N_3)$ with 
\begin{align}
M^{-(1+\de)}\le N_{1},N_{2},N_{3}\le M^{\frac{2}{n}}.\label{eq:condi-supp}
\end{align}
Let us further localize $\sigma$ variable with respect to $L'\in2^{\Z}$ to write 
\begin{align*}
&\mathcal{K}_{\textbf{L},\mathbf{N}}(s,\xi) = \\
&\; \iint_{\R^{2}\times\R^{2}}e^{isq\freq}\rho_{L}(\eta)\rho_{L'}(\sigma)|\eta|^{-1}\wh{f_{N_{1}}}(s,\xi+\eta)\wh{f_{N_{2}}}(s,\xi+\sigma)\overline{\wh{f_{N_{3}}}(s,\xi+\eta+\sigma)}d\eta d\sigma,
\end{align*}
where we denoted $\textbf{L}=(L,L')\in (2^{\Z})^2$. 
Then, \eqref{eq:goal-high} reduces to showing that 
\begin{align}
    \sum_{\supp\in\mathcal{A}}\left|\mathcal{K}_{\textbf{L},\textbf{N}}(s,\xi)\right|\les\ve_{1}^{3}M^{-(1+\de)}\bra{N}^{-k},
\label{eq:goal2-high} 
\end{align}
where the summation runs over 
\begin{align}\begin{aligned}\label{eq:condi-piece}
    \mathcal{A}=\Big\{  \left(\textbf{L},\textbf{N}\right)\in (2^{\Z})^5 &: M^{-(1+\de)}\le N_{1},N_{2},N_{3}\le M^{\frac{2}{n}},\; \max(N_1,N_2,N_3)\sim N, \\ 
   &\hspace{3cm} L_{0}\le L \le \max(N_2,N_3), \; L'\le \max(N_1,N_3)\Big\} .
\end{aligned}\end{align}
By H\"older inequality, we also readily have 
\begin{align*}
    \left|\mathcal{K}_{L,\mathbf{N}}(s,\xi)\right| &\les
    L^{-1}\|\rho_{L}\|_{L^{1}}\|\rho_{L'}\|_{L^{1}}\langle N_1\rangle^{-k}\langle N_2\rangle^{-k}\langle N_3\rangle^{-k}\normo{\bra{\xi}^{k}\wh{f}(s,\xi)}_{L_{\xi}^{\infty}}^3 \\ 
    &\les L(L')^{2}\bra{N}^{-k}\normo{\bra{\xi}^{k}\wh{f}(s,\xi)}_{L_{\xi}^{\infty}}^3.
\end{align*}
Hence we may obtain the desired estimates \eqref{eq:goal2-high} 
whenever the summation runs over those indexes in $\mathcal{A}$ satisfying $L(L')^{2}\le M^{-(1+2\de)}$. Indeed, we have
\global\long\def\nl{{\textbf{L},\textbf{N}}}%
\begin{align*}
    \sum_{\mathcal{A}\cap\left\{ L(L')^{2}\le M^{-(1+2\de)}\right\}}|\mathcal{K}_{\nl}(s,\xi)|  
    \les \ve_{1}^{3}M^{-(1+\de)}\bra{N}^{-k}.
\end{align*}
We are left with a summation over 
\begin{align*}
    \mathcal{B}=\mathcal{A}\cap\left\{ L(L')^{2}\ge M^{-(1+2\de)}\right\}.
\end{align*}
We perform an integration by parts in $\eta$ twice, with a slight abuse of notation, to obtain 
\begin{align*}  
&|\mathcal{K}_{\nl}(s,\xi)|  \le |\mathcal{K}_{1}(s,\xi)|+|\mathcal{K}_{2}(s,\xi)|+|\mathcal{K}_{3}(s,\xi)|,\\
 &   \mathcal{K}_{1}(s,\xi)  :=\frac{1}{s^{2}}\int\!\!\!\!\int e^{isq\freq}\textbf{m}_{1}(\xi,\eta,\sigma)\nabla_{\eta}^{2}\left(\wh{f_{N_{1}}}(s,\xi+\eta)\overline{\wh{f_{N_{3}}}(s,\xi+\eta+\sigma)}\right)\wh{f_{N_{2}}}(s,\xi+\sigma)d\eta d\sigma, \\
  &      \mathcal{K}_{2}(s,\xi)  :=\frac{1}{s^{2}}\int\!\!\!\!\int e^{isq\freq}\textbf{m}_{2}(\xi,\eta,\sigma)\nabla_{\eta}\left(\wh{f_{N_{1}}}(s,\xi+\eta)\overline{\wh{f_{N_{3}}}(s,\xi+\eta+\sigma)}\right)\wh{f_{N_{2}}}(s,\xi+\sigma)d\eta d\sigma,\\
   & \mathcal{K}_{3}(s,\xi) :=\frac{1}{s^{2}}\int\!\!\!\!\int e^{isq\freq}\textbf{m}_{3}(\xi,\eta,\sigma)\wh{f_{N_{1}}}(s,\xi+\eta)\overline{\wh{f_{N_{3}}}(s,\xi+\eta+\sigma)}\wh{f_{N_{2}}}(s,\xi+\sigma)d\eta d\sigma, 
    \end{align*}
    where 
    \begin{align*}
        \textbf{m}_{1}(\xi,\eta,\sigma) & :=\left( \frac{\nabla_{\eta}q\freq}{|\nabla_{\eta}q\freq|^{2}}\right)^2|\eta|^{-1}\rho_{L}(\eta)\rho_{L'}(\sigma)\rho_{N_{1}}(\xi+\eta)\rho_{N_{3}}(\xi+\eta+\sigma),\\  
    \textbf{m}_{2}(\xi,\eta,\sigma) & :=\frac{\nabla_{\eta}q\freq}{|\nabla_{\eta}q\freq|^{2}}\nabla_{\eta}\left(\frac{\nabla_{\eta}q\freq}{|\nabla_{\eta}q\freq|^{2}}|\eta|^{-1}\rho_{L}(\eta)\rho_{N_{1}}(\xi+\eta)\rho_{N_{3}}(\xi+\eta+\sigma)\right)\rho_{L'}(\sigma),\\ 
    \textbf{m}_{3}(\xi,\eta,\sigma) & :=\nabla_{\eta}\textbf{m}_{2}\freq.
    \end{align*}
A direct computation yields that 
\begin{align}
    \begin{aligned}
    \label{lower bound of nabla q}
    \left|\nabla_{\eta}q(\xi,\eta,\sigma)\right|  
    &=\left|\frac{\xi+\eta+\sigma}{\bra{\xi+\eta+\sigma}}-\frac{\xi+\eta}{\bra{\xi+\eta}}\right| \\ 
    &\gtrsim \frac{|\sigma|}{\max(\langle \xi+\eta+\sigma\rangle, \langle\xi+\eta\rangle) \min(\langle \xi+\eta+\sigma\rangle, \langle\xi+\eta\rangle)^2}.
\end{aligned}\end{align}
With the help of \eqref{lower bound of nabla q}, one can verify that the symbols satisfy the following bounds:
\begin{align}
\left\Vert \iint_{\R^{2}\times\R^{2}}\textbf{m}_{1}\freq e^{iy\cdot\eta}e^{iz\cdot\sigma} d\eta d\sigma\right\Vert _{L_{y,z}^{1}(\R^{2}\times\R^{2})} & \les  L^{-1}(L')^{-2}\langle\max(N_1,N_3)\rangle\langle\min(N_1,N_3)\rangle^{10},\label{eq:k_1 bound} \\ 
\normo{\iint_{\R^{2}\times\R^{2}}\textbf{m}_{2}(\xi,\eta,\sigma)e^{iy\cdot\eta}e^{iz\cdot\sigma}\,d\sigma d\eta}_{L_{y,z}^{1}(\R^{2}\times\R^{2})}
&\les L^{-2}(L')^{-2}\langle\max(N_1,N_3)\rangle^2\langle\min(N_1,N_3)\rangle^{12}, \label{eq:k_2 bound}
\end{align}
and 
\begin{align}
    \begin{aligned}\label{eq:pw-bound-k3}
    \left|\textbf{m}_{3}\freq\right| & \les L^{-3}(L')^{-2}\langle\max(N_1,N_3)\rangle^2\langle\min(N_1,N_3)\rangle^{4}.
    \end{aligned}
    \end{align} 
By applying the operator inequality \eqref{eq:coif-1-1} with \eqref{eq:k_1 bound}, we estimate 
\begin{align*}
 & \sum_{(\nl)\in\mathcal{B}}\left|\mathcal{K}_{1}(s,\xi)\right|\\
 & \les M^{-2}\sum_{(\nl)\in\mathcal{B}}L^{-1}(L')^{-2}\langle\max(N_1,N_3)\rangle\langle\min(N_1,N_3)\rangle^{10} \\ 
 &\qquad\qquad \qquad\qquad\times \normo{P_{N_{2}}u}_{L_{x}^{\infty}}\left( \normo{x^{2}f}_{L^{2}}\normo{f_{N_{3}}}_{L^{2}} + \normo{xf}_{L^{2}}^2 + \normo{x^{2}f}_{L^{2}}\normo{f_{N_{1}}}_{L^{2}}\right)\\
 & \les \ve_1^3 M^{-2+2\delta+2\delta_0}\sum_{ \substack{M^{-(1+\de)}\le N_{1},N_{2},N_{3}\le M^{\frac{2}{n}} \\   L_0\le L \le \max(N_2,N_3)}}\langle\max(N_1,N_3)\rangle\langle\min(N_1,N_3)\rangle^{10} \langle N_2\rangle^{-k}\\
 &\les  \ve_1^3 M^{-2+3\delta+2\delta_0} \bra{N}^{12}  \les\ve_{1}^{3}M^{-(1+\de)}\bra{N}^{-k}.
\end{align*}
Similarly, \eqref{eq:k_2 bound} implies that
\begin{align*}
 & \sum_{(\nl)\in\mathcal{B}}\left|\mathcal{K}_{2}(s,\xi)\right|\\
 & \les M^{-2}\sum_{(\nl)\in\mathcal{B}}L^{-2}(L')^{-2}\langle\max(N_1,N_3)\rangle^2\langle\min(N_1,N_3)\rangle^{12} \\ 
 &\qquad\qquad\qquad\qquad\qquad\qquad\qquad\qquad\times \normo{P_{N_{2}}u}_{L^{\infty}}\normo{xf}_{L^{2}}\left( \normo{f_{N_{3}}}_{L^{2}} + \normo{f_{N_{1}}}_{L^{2}}\right)\\
 & \les \ve_1^3 M^{-2+\delta_0}\sum_{ \substack{M^{-(1+\de)}\le N_{1},N_{2},N_{3}\le M^{\frac{2}{n}} \\   L_0\le L \le \max(N_2,N_3)}}L^{-1}\langle\max(N_1,N_3)\rangle^2 \langle\min(N_1,N_3)\rangle^{12} \langle N_2\rangle^{-k}\\
 &\hspace{9cm}\times\Big(N_1^{\frac12}\langle N_1\rangle^{-k}+N_3^{\frac12}\langle N_3\rangle^{-k}\Big ) \\
  & \les \ve_1^3 M^{-\frac{11}{10}+\delta_0}\sum_{ M^{-(1+\de)}\le N_{1},N_{2},N_{3}\le M^{\frac{2}{n}} }\langle\max(N_1,N_3)\rangle^2 \langle\min(N_1,N_3)\rangle^{12} \langle N_2\rangle^{-k}\\
  &\hspace{9cm}\times \Big(N_1^{\frac12}\langle N_1\rangle^{-k}+N_3^{\frac12}\langle N_3\rangle^{-k}\Big ) \\
 & \les \ve_1^3 M^{-\frac{11}{10}+\delta_0}\langle N\rangle^{14}  \les\ve_{1}^{3}M^{-(1+\de)}\bra{N}^{-k}.
 \end{align*}
Finally, using the pointwise bound in \eqref{eq:pw-bound-k3}, we estimate 
\begin{align*}
    & \sum_{(\nl)\in\mathcal{B}}\left|\mathcal{K}_{3}(s,\xi)\right|\\
    & \les M^{-2}\sum_{(\nl)\in\mathcal{B}}L^{-3}(L')^{-2}\langle\max(N_1,N_3)\rangle^2\langle\min(N_1,N_3)\rangle^{4}\|\rho_{L}\|_{L^{1}}\|\rho_{L'}\|_{L^{1}}
    \normo{\wh{f_{N_{1}}}}_{L^{\infty}}\normo{\wh{f_{N_{2}}}}_{L^{\infty}}\normo{\wh{f_{N_{3}}}}_{L^{\infty}}\\
    & \les\ve_{1}^{3}M^{-2}\sum_{(\nl)\in\mathcal{B}}L^{-1}\langle\max(N_1,N_3)\rangle^2\langle\min(N_1,N_3)\rangle^{4}\bra{N_{1}}^{-k}\bra{N_{2}}^{-k}\bra{N_{3}}^{-k}\\
    & \les\ve_{1}^{3}M^{-\frac{11}{10}}\sum_{ \substack{M^{-(1+\de)}\le N_{1},N_{2},N_{3}\le M^{\frac{2}{n}} \\ L'\sim \max(N_1,N_3)} }\langle\max(N_1,N_3)\rangle^2\langle\min(N_1,N_3)\rangle^{4}\bra{N_{1}}^{-k}\bra{N_{2}}^{-k}\bra{N_{3}}^{-k}\\
    & \les\ve_{1}^{3}M^{-(1+\de)}\bra{N}^{-k}.
   \end{align*}

\subsection*{Acknowledgement}

The first and second authors were supported by the National Research
Foundation of Korea(NRF) grant funded by the Korea government(MSIT)
(No. NRF-2019R1A5A1028324). The first author was supported in part
by NRF-2022R1A2C1091499. The second author was supported in part by
NRF-2022R1I1A1A0105640812. The third author was supported in part by
NRF-2021R1C1C1005700. 

%%%%%%%%%%%%%%%%%%%%%%%%%%%%%%%%%%%%%%%%%%%%%%%%%%%%%%%%%%%%%%%%%%%%%%%%%%%%%%%%%%%%%%%%%%%%%%%%%%%%%%%%%%%%%%%%%%%%%%%%

 \bibliographystyle{plain}

\medskip{}

\end{document}